\documentclass[a4paper,10pt,reqno]{amsart}
\usepackage{amsmath,amsfonts,amscd,amsthm,amssymb,graphicx,mathrsfs,eufrak,mathrsfs}

\usepackage[nospace,noadjust]{cite}
\usepackage[dvips,all,arc,curve,color,frame]{xy}
\usepackage{tikz}
\usetikzlibrary{arrows,decorations.pathmorphing,decorations.pathreplacing,positioning,shapes.geometric,shapes.misc,decorations.markings,decorations.fractals,calc,patterns}
\usepackage[colorlinks]{hyperref}
\usepackage{comment}

\tikzset{>=stealth',
        cvertex/.style={circle,draw=black,inner sep=1pt,outer sep=3pt},
        vertex/.style={circle,fill=black,inner sep=1pt,outer sep=3pt},
        star/.style={circle,fill=yellow,inner sep=0.75pt,outer sep=0.75pt},
        tvertex/.style={inner sep=1pt,font=\scriptsize},
        gap/.style={inner sep=0.5pt,fill=white}}

\usepackage{mathtools}

\author{J. Karmazyn}
\title{Deformations of Algebras Defined by Tilting Bundles}

\theoremstyle{definition} 
\newtheorem{resolution}{Definitions}[section]
\newtheorem{TiltingDef}[resolution]{Definition}

\newtheorem{Nilpotent Deformations}[resolution]{Definition}
\newtheorem{Graded Deformations}[resolution]{Definition}
\newtheorem{PBW Deformation}[resolution]{Definition}
\newtheorem{Scheme Deformations}[resolution]{Definitions}
\newtheorem{Algebra Deformations}[resolution]{Definitions}
\newtheorem{Formal Deformations}[resolution]{Definition}
\newtheorem{Graded Scheme Deformations}[resolution]{Definition}
\newtheorem{Graded Algebra Deformations}[resolution]{Definition}

\newtheorem{Graded Conditions}[resolution]{Conditions}
\newtheorem{RCA}[resolution]{Definition}

\newtheorem{RCA[3,2]}[resolution]{Example}

\newtheorem{GoodAction}[resolution]{Definition}

\newtheorem{C*Action}[resolution]{Definition}
\newtheorem{LiftBundle}[resolution]{Definition}
\newtheorem{GradedDefinitions}[resolution]{Definitions}
\newtheorem{TiltingRemark}[resolution]{Remark}
\newtheorem{RemarkTU}[resolution]{Remark}
\newtheorem{GradingActionRemark}[resolution]{Remark}

\newtheorem{ActionGenerationRemark}[resolution]{Remark}
\newtheorem{CompleteActionDefinition}[resolution]{Definition}
\newtheorem{GradedRingRemarks}[resolution]{Remark}
\newtheorem{ShiftedGradingRemark}[resolution]{Remark}
\newtheorem{NonUniqueGradingRemark}[resolution]{Remark}
\newtheorem{GradedReconstructionDefinition}[resolution]{Definition}
\newtheorem{RCAExample}[resolution]{Example}
\newtheorem{Reconstruction Algebra Example 12,7}[resolution]{Example}
\newtheorem{ExampleSimpleTypeA}[resolution]{Example}
\newtheorem{ReesFibresRemark}[resolution]{Remark}
\newtheorem{TypeDExample}[resolution]{Example}
\newtheorem{NonQuotientExample}[resolution]{Example}
\newtheorem{ExplicitCalculationRemark}[resolution]{Remark}

\theoremstyle{plain} 

\newtheorem{Tilting}[resolution]{Theorem}

\newtheorem{Reconstruction}[resolution]{Theorem}

\newtheorem{TiltingDeform}[resolution]{Lemma}

\newtheorem{Main}[resolution]{Theorem}
\newtheorem{Graded Main}[resolution]{Theorem}
\newtheorem*{PMain}{Theorem}
\newtheorem*{PMainGraded}{Theorem}

\newtheorem{DualPullback}[resolution]{Lemma}
\newtheorem{SymplecticTiltingBundle}[resolution]{Lemma}
\newtheorem{LiftingToEquiTilting}[resolution]{Proposition}

\newtheorem{GradedReconstruction}[resolution]{Corollary}
\newtheorem{FlatBaseChange}[resolution]{Lemma}

\newtheorem{GL2Main}[resolution]{Corollary}
\newtheorem{RationalMain}[resolution]{Corollary}
\newtheorem{GL2AllDeforms}[resolution]{Lemma}

\newtheorem{FilteredCorollary}[resolution]{Corollary}

\newtheorem{KaledinGinzburg}[resolution]{Lemma}
\newtheorem{SymplectRef}[resolution]{Corollary}

\newtheorem{HHSQS}[resolution]{Lemma}

\newtheorem{VitalLemma1}[resolution]{Lemma}
\newtheorem{VitalLemma1'}[resolution]{Lemma}
\newtheorem{GradedVitalLemma1}[resolution]{Lemma}

\newtheorem{GradedTiltingDeform}[resolution]{Lemma}
\newtheorem{GradedGeneration}[resolution]{Lemma}

\newtheorem{GradedTiltingBundles}[resolution]{Lemma}

\newtheorem{ReesRingGradedDeformations}[resolution]{Lemma}

\newtheorem{TodaUehara}[resolution]{Lemma}
\newtheorem{DerivedBaseChange}[resolution]{Lemma}

\newtheorem{Equivariant Ample Line Bundle}[resolution]{Lemma}

\newtheorem{GradedNakayama}[resolution]{Lemma}

\newtheorem{BaseChangeCorollary}[resolution]{Corollary}
\newtheorem{BetterGeneration}[resolution]{Lemma}

\newtheorem{ActionLemma}[resolution]{Lemma}
\newtheorem{GradedBaseChangeCorollary}[resolution]{Corollary}
\newtheorem{Bounded Push Down}[resolution]{Lemma}

\newtheorem*{pRationalMain}{Corollary}
\newtheorem*{pGL2Main}{Corollary}
\newtheorem*{pSymplectRef}{Corollary}
\newtheorem*{PFilteredCorollary}{Corollary}

\newtheorem{IntroFibreCorollary}[resolution]{Corollary}
\newtheorem{IntroFibreGradedCorollary}[resolution]{Corollary}

\DeclareMathAlphabet{\mathbbm}{U}{bbm}{m}{n}

\let\oldtocsection=\tocsection
\let\oldtocsubsection=\tocsubsection
\let\oldtocsubsubsection=\tocsubsubsection
\renewcommand{\tocsection}[2]{\hspace{0em}\oldtocsection{#1}{#2}}
\renewcommand{\tocsubsection}[2]{\hspace{1em}\oldtocsubsection{#1}{#2}}
\renewcommand{\tocsubsubsection}[2]{\hspace{2em}\oldtocsubsubsection{#1}{#2}}

\newcommand{\Hom}{\textnormal{Hom}}
\newcommand{\Perf}{\textnormal{Perf}}
\newcommand{\SHom}{\mathcal{H}{om}}
\newcommand{\Ext}{\textnormal{Ext}}
\newcommand{\Tor}{\textnormal{Tor}}
\newcommand{\Coh}{\textnormal{Coh}}
\newcommand{\QCoh}{\textnormal{QCoh}}

\newcommand{\End}{\textnormal{End}}
\newcommand{\CalEnd}{\mathcal{E}nd}

\newcommand{\HH}{\textnormal{HH}}
\newcommand{\CH}{\textnormal{H}}
\newcommand{\SL}{\textnormal{SL}}
\newcommand{\GL}{\textnormal{GL}}
\newcommand{\Sp}{\textnormal{Sp}}
\newcommand{\Spec}{\textnormal{Spec}}

\newcommand{\Hilb}{\textnormal{Hilb}}

\def\mod{\mathop{\textnormal{mod}}\nolimits}

\newcommand{\RCalHom}{\mathcal{RH}\textnormal{om}}
\newcommand{\Aut}{\textnormal{Aut}}
\newcommand{\Sym}{\textnormal{Sym}}

\newcommand{\E}{\mathcal{E}}
\newcommand{\CalHom}{\mathcal{H}\textnormal{om}}

\begin{document}

\maketitle
\begin{abstract}

In this paper we produce noncommutative algebras derived equivalent to deformations of schemes with tilting bundles. We do this in two settings, first proving that a tilting bundle on a scheme lifts to a tilting bundle on an infinitesimal deformations of that scheme and then expanding this result to $\mathbb{C}^*$-equivariant deformations over schemes with a good $\mathbb{C}^*$-action. In both these situations the endomorphism algebra of the lifted tilting bundle produces a deformation of the original endomorphism algebra, and this is a graded deformation in the $\mathbb{C}^*$-equivariant case.

We apply our results to rational surface singularities, generalising the deformed preprojective algebras of \cite{CBH}, and also to symplectic situations where the deformations produced are related to symplectic reflection algebras.
\end{abstract}

\tableofcontents
\addcontentsline{toc}{chapter}{Contents}

\section{Introduction}
\subsection{Overview}

A tilting bundle $T$ on a scheme $X$ induces an equivalence between the derived category of quasicoherent sheaves on $X$ and the derived category of right modules for the algebra $A:=\End_{X}(T)$. This is a valuable tool which allows homological properties to be understood and calculated in either the algebraic or geometric setting.  Deformation theory is controlled by homological data, and it is expected that if a scheme is derived equivalent to an algebra then its deformations should be derived equivalent to certain deformations of the algebra. It is anticipated that a tilting bundle should lift to deformations of the scheme and that the endomorphism algebra of the lifted tilting bundle should produce deformations of the algebra \cite[Section 3.8]{BHTilting}. We will prove this result both for deformations over complete local Noetherian schemes and for $\mathbb{C}^*$-equivariant deformations over affine schemes with unique $\mathbb{C}^*$-invariant closed point.

These results can be viewed from two intertwined perspectives. On the one hand they produce noncommutative algebras derived equivalent to geometric deformations, and so construct noncommutative counterparts to interesting geometric phenomenon. One example is the simultaneous resolution of the \emph{Artin component} for quotient surface singularities. In this situation a quotient surface singularity has a versal deformation, and the Artin component of this has a simultaneous resolution after base change which can be realised as the versal deformation of the minimal resolution.  The existence of a derived equivalent noncommutative algebra for quotient singularities $\mathbb{C}^2/G$ with $G< \GL_2(\mathbb{C})$ was proposed by Riemenschneider \cite{RiemenschneiderMcKay}, and we discuss the construction of such an algebra as one of our applications below. Further, realising these algebras as deformations makes it possible to actually calculate them in many examples. 

On the other hand, the results construct deformations of noncommutative algebras via geometric techniques, with the advantage of finding exactly the deformations of noncommutative algebras that correspond to geometric deformations. For $G$ a finite subgroup of $\SL_n(\mathbb{C})$ the skew group algebras $\mathbb{C}[x_1, \dots, x_n] \rtimes G$ always provide noncommutative crepant resolutions of the quotient singularities $\mathbb{C}^n/G$ \cite{NCCRs}, and as Koszul algebras their PBW deformations are classified by the results of Braverman and Gaitsgory \cite{BrG}. However, considering the case of a small, finite subgroup of $\GL_2(\mathbb{C})$ then minimal resolutions exist but if $G$ is not a subgroup of $\SL_2(\mathbb{C})$ then the skew group algebras are not derived equivalent to the minimal resolutions and have no PBW deformations \cite[Theorem 5.0.6.]{CYPBW} despite there being non-trivial deformations of both the corresponding quotient singularity and minimal resolution.

In these cases a better noncommutative resolution is the reconstruction algebra, which is defined as the endomorphism algebra of a specific tilting bundle on the minimal resolution \cite{WemyssGL2}. However, these algebras are not Koszul and as such there is a lack of available technology to calculate their graded deformations, and in general their graded deformations are unknown. Previous attempts have been made to understand the deformations in specific examples via diamond lemma techniques, \cite{StefanDavidProject}, and these highlight the difficulty of calculating deformations without the Koszul deformation theorems. Our result works outside of the Koszul setting, lifting certain $\mathbb{C}^*$-equivariant deformations of the geometry to graded deformations of the algebra via the tilting bundle. This produces precisely the graded deformations of the algebra which are derived equivalent to commutative deformations of the geometry, and so we are able to calculate certain deformations of the reconstruction algebras as one of our applications.

\subsection{Lifting tilting bundles to deformations} We give an overview of the main results. The first result we will prove is the following, where the definitions of deformations of schemes and algebras are recalled in Section \ref{Deformations of Schemes}.

\begin{PMain}[Theorem \ref{Main}] \label{PMain} Suppose that $p:X_0 \rightarrow \Spec(R_0)$ is a projective morphism of  Noetherian schemes, that $\left( \rho:X \rightarrow \Spec\,D, j_d:X_0 \rightarrow X \right)$ is a deformation of $X_0$ over a complete local Noetherian $\mathbb{C}$-algebra $(D,d)$, and suppose that $\rho$ factors through a projective morphism $\pi:X \rightarrow \Spec(R)$ where $R$ is flat over $D$ and $R \otimes_D D/d \cong R_0$. 

Then any tilting bundle $T_0$ on $X_0$ lifts uniquely to a tilting bundle $T$ on $X$. Further, if $A_0=\End_{X_0}(T_0)$ then  $A:=\End_{X}(T)$ is a $D$-algebra which is flat as a $D$-module, there is a map $A \rightarrow A_0$ such that $A \otimes_{D} D/d \cong A_0$, and this map and isomorphism are uniquely defined by $j_d$, so the deformation of $X_0$ defines a deformation of $A_0$. 
\end{PMain}
In particular, a deformation of $X_0$ consists of a flat morphism $\rho:X \rightarrow \Spec(D)$ and an identification of $X_0$ with the fibre of $\rho$ over $d \in \Spec(D)$ provided by the morphism $j_d:X_0 \rightarrow X$.
\begin{IntroFibreCorollary}Suppose that $(D,d)$ is a complete local Noetherian ring and $\rho:X \rightarrow \Spec(D)$ is a flat morphism of Noetherian schemes that factors through a projective morphism $\pi:X \rightarrow \Spec(R)$ and flat affine morphism $\alpha:\Spec(R) \rightarrow \Spec(D)$. If there is a tilting bundle $T_d$ on the fibre $X_d$ of $X$ over $d$, then there is a tilting bundle $T$ on $X$ such that $\End_X(T)$ is flat over $\Spec(D)$ and $\End_X(T) \otimes_D D/d \cong \End_{X_d}(T_d)$.
\end{IntroFibreCorollary}

This theorem relies on the assumption that the deformation is over a scheme with a unique closed point, restricting us to considering  deformations over Artinian local or Noetherian complete local $\mathbb{C}^*$-algebras. There is a natural extension which maintains a unique privileged point in a non-local setting: working with deformations over finite type schemes with $\mathbb{C}^*$-actions and a unique $\mathbb{C}^*$-fixed closed point.

We then prove the $\mathbb{C}^*$-equivariant extension of the above result, noting that $\mathbb{C}^*$-equivariant deformations of a scheme and graded deformations of algebras are defined in Section \ref{GradedDeformationsandSchemes}.

\begin{PMainGraded}[Theorem \ref{Graded Main}] \label{GradedPMain} 

Suppose $p:X_0 \rightarrow \Spec(R_0)$ is a $\mathbb{C}^*$-equivariant projective morphism of finite type schemes with $\mathbb{C}^*$-actions, that $X_0$ is normal, and suppose that the $\mathbb{C}^*$-action on $\Spec(R_0)$ has a unique fixed closed point.  Let $\left(\rho:X \rightarrow \Spec \, D, j_d:X_0 \rightarrow X \right)$ be a $\mathbb{C}^*$-equivariant deformation of $X_0$ over $(D,d)$ where $(D,d)$ is a graded complete local Noetherian ring with $\mathbb{C}^*$-action on $\Spec(D)$ having unique invariant closed point, and such that $\rho$ factors through a projective morphism $\pi:X \rightarrow \Spec(R)$ where $R$ is flat over $D$ with good $\mathbb{C}^*$-action and $R \otimes_D D/d \cong R_0$.

Then any $\mathbb{C}^*$-equivariant tilting bundle $T_0$ on $X_0$ lifts to a $\mathbb{C}^*$-equivariant tilting bundle $T$ on $X$. Further, if $A_0=\End_{X_0}(T_0)$ then  $A:=\End_{X}(T)$ is a graded $D$-algebra which is flat and graded as a $D$-module, there is a graded map $A \rightarrow A_0$ such that $A \otimes_{D} D/d \cong A_0$, and this map and isomorphism are uniquely defined by $j_d$ so the $\mathbb{C}^*$-equivariant deformation of $X_0$ defines a graded deformation of $A_0$. 
\end{PMainGraded}
This gives the following corollary.
\begin{IntroFibreGradedCorollary} Suppose that $(D,d)$ is a complete graded local Noetherian ring and $\rho:X \rightarrow \Spec(D)$ is a $\mathbb{C}^*$-equivariant flat morphism of Noetherian schemes that factors through a $\mathbb{C}^*$-equivariant projective morphism $\pi:X \rightarrow \Spec(R)$ and $\mathbb{C}^*$-equivariant flat affine morphism $\alpha:\Spec(R) \rightarrow \Spec(D)$. If there is a $\mathbb{C}^*$-equivariant tilting bundle $T_d$ on the fibre $X_d$ of $X$ over $d$ then there is a $\mathbb{C}^*$-equivariant tilting bundle $T$ on $X$ such that $\End_X(T)$ is flat over $\Spec(D)$  and $\End_X(T) \otimes_D D/d \cong \End_{X_d}(T_d)$.
\end{IntroFibreGradedCorollary}

An advantage of graded deformations is that there may be closed points in $\Spec(D)$ other than the unique invariant point $d$, and we may consider fibres over these points.

\begin{PFilteredCorollary}[Corollary \ref{FilteredCorollary}] \label{PFilteredCorollary} Suppose that $A_0$ is positively graded ring, and suppose that $j_z: z \hookrightarrow \Spec(D)$ is the inclusion of a closed point. Then $T_z:=j_z^*T$ is a tilting bundle on the fibre $X_z:=X \times_{\Spec \, D} z$ such that $A_z:=\End_{X_z}(T_z) \cong A \otimes_D D/z$ is a filtered algebra with associated graded algebra isomorphic to $A_0$.
\end{PFilteredCorollary}

In this situation many properties of the fibres $A_z$ can be deduced from $A$ via the associated graded property without having to understand the larger algebra $A$ itself.

\subsection{Applications}
We briefly outline the applications we will consider in Section \ref{Applications}. In the ungraded setting we consider minimal resolutions of rational surfaces singularities, and in the graded setting we consider both surface quotient singularities and symplectic quotient singularities.

\subsubsection*{Minimal resolutions of rational surface singularities and reconstruction algebras} 

There has been significant interest in the deformation theory of minimal resolutions of rational surface singularities since the work of \cite{Artin} with the aim of understanding these versal deformations as simultaneous resolutions of deformations of the singularity \cite{Riemanschneider,WahlSimultaneous}, and explicitly calculating certain examples \cite{RiemenschneiderNach,BehnkeRiemenschneider}. The minimal resolutions are known to carry tilting bundles and to be derived equivalent to the reconstruction algebras \cite{WemyssGL2}. 

The results of this paper produce a tilting bundle on the versal deformation of the minimal resolution and define a deformation of the corresponding reconstruction algebra derived equivalent to the geometric deformation.

\begin{pRationalMain}[Corollary \ref{RationalMain}] Let $\Spec(R_0)$ be a complete local, rational, surface singularity, $p:X_0 \rightarrow \Spec(R_0)$ be its minimal resolution, and $A_0=\End_{X_0}(T_0)$ its corresponding reconstruction algebra. Consider the versal deformation $(\rho:X \rightarrow \Spec \, D, j_d: X_0 \rightarrow X)$ of $X_0$. Then:
\begin{enumerate}
\item The tilting bundle $T_0$ on $X_0$ lifts to a tilting bundle $T$ on the versal deformation space  $X$ of $X_0$. 
\item The algebra $A:=\End_{X}(T)$ is derived equivalent to the versal deformation space $X$.
\item The algebra $A$ is a deformation of $A_0$ over $D$ induced by the deformation $(\rho:X \rightarrow \Spec \, D, j_d:X_0 \rightarrow X)$. 
\end{enumerate}
\end{pRationalMain}

The existence of a noncommutative analogue to the simultaneous resolution of the Artin component was proposed by Riemenschneider \cite{RiemenschneiderMcKay}. This result produces this algebra as $A$ and shows that it occurs as a deformation of the reconstruction algebra and is produced via a tilting bundle. 

\subsubsection*{Minimal resolutions of surface quotient singularities and graded reconstruction algebras}

If we restrict the previous case to surface quotient singularities then the geometry has a natural $\mathbb{C}^*$-equivariant structure and the reconstruction algebra is naturally graded as a sub-algebra of the corresponding McKay quiver with relations.  
An application of our results will be to construct graded deformations of the reconstruction algebras, which will be derived equivalent to the versal deformation of the minimal resolutions in this non-local setting.  

\begin{pGL2Main}[Corollary \ref{GL2Main}] Let $G$ be a small, finite subgroup of $\GL_2(\mathbb{C})$, let $\mathbb{C}^2/G$ be the corresponding quotient singularity, let $p:X_0 \rightarrow \mathbb{C}^2/G$ be the $\mathbb{C}^*$-equivariant minimal resolution, and let $A_0=\End_{X_0}(T_0)$ be the corresponding graded reconstruction algebra. Consider the $\mathbb{C}^*$-equivariant versal deformation $(X \rightarrow \Spec \, D,j_d:X_0 \rightarrow X)$ of $X_0$. Then:
\begin{enumerate}
\item The $\mathbb{C}^*$-equivariant tilting bundle $T_0$ on $X_0$ lifts to a $\mathbb{C}^*$-equivariant tilting bundle $T$ on $X$.
\item The algebra $A:=\End_X(T)$ is derived equivalent to the versal deformation space $X$ of $X_0$.
\item The algebra $A$ is a graded deformation of $A_0$ over $D$ induced by the $\mathbb{C}^*$-equivariant versal deformation $(X \rightarrow \Spec \, D, j_d:X_0\rightarrow X)$.
\item For any closed point $z \in \Spec(D)$ the fibres $A_z:=A \otimes_D (D/z)$ are filtered algebras with the property that the associated graded algebra of $A_z$ is isomorphism to $A_0$.
\end{enumerate}
\end{pGL2Main}

When $G<\SL_2(\mathbb{C})$ then the fibres of $A$ are deformed preprojective algebras as defined by Crawley-Boevey and Holland \cite{CBH}. More generally the fibres of this graded reconstruction algebras give an analogue of the deformed preprojective algebras outside of a Koszul setting, generalising from the Kleinian singularity case to all surface quotient singularities.

The preprojective algebras have an additional deformation parameter which does not correspond to any geometric deformation and so cannot be constructed by Corollary \ref{GL2Main}. However, it is shown in Section \ref{NonGeometric} that there are no non-geometric graded deformations of a graded reconstruction algebra when $G \nless \SL_2(\mathbb{C})$. This is a major difference between the noncommutative deformations of the preprojective and reconstruction algebras, and this shows that Corollary \ref{GL2Main} constructs all graded deformations of the reconstruction algebras corresponding to $G\nless \SL_2(\mathbb{C})$ quotient surface singularities.

\subsubsection*{Crepant resolutions of symplectic quotient singularities and symplectic reflection algebras}
While symplectic resolutions of quotient singularities do not always exist, when they do they are known to have $\mathbb{C}^*$-equivariant tilting bundles with endomorphism algebras the skew group algebras \cite{SympMcKay}, and to have an interesting $\mathbb{C}^*$-equivariant deformation known as the versal Poisson or Calogero-Moser deformation \cite{NamikawaFlopsandPoisson,PoisDefKaledinGinzburg}. As such they provide a natural application of our results.

\begin{pSymplectRef}[Corollary \ref{SymplectRef}]
Let $G$ be a finite subgroup of $Sp_{2n}(\mathbb{C})$, let $\mathbb{C}^{2n}/G$ be the corresponding quotient singularity, and suppose that a crepant resolution $p:X_0 \rightarrow \mathbb{C}^{2n}/G$ exists. Consider the $\mathbb{C}^*$-equivalent tilting bundle $T_0$ constructed in \cite{SympMcKay} such that $\End_{X_0}(T_0)$ is the skew group algebra and let $(\rho:X \rightarrow \Spec \,D ,j_d:X_0 \rightarrow X)$ be the $\mathbb{C}^*$-equivariant versal Poisson deformation. Then:
\begin{enumerate}
\item The $\mathbb{C}^*$-equivariant tilting bundle $T_0$ on $X_0$ lifts to a $\mathbb{C}^*$-equivariant tilting bundle $T$ on $X$.
\item The algebra $A:=\End_{X}(T)$ is derived equivalent to the $\mathbb{C}^*$-equivariant versal Poisson deformation space $X$ of $X_0$.
\item The algebra $A$ is a graded deformation of the skew group algebra over $D$ induced by the $\mathbb{C}^*$-equivariant versal deformation $(X \rightarrow \Spec \, D, j_d:X_0 \rightarrow X)$.
\item For any closed point $z \in \Spec(D)$ the fibres $A_z:=A \otimes_D (D/z)$ are symplectic reflection algebras with noncommutative parameter $t$ set equal to 0.
\end{enumerate}
\end{pSymplectRef}

The symplectic reflection algebras are defined as fibres of graded deformations of the skew group algebras \cite{EG}. We note that when $G$ is a wreath product group a crepant resolution exists and these sympletic reflection algebras were already known to  be produced from tilting bundles by the results of Gordon and Smith \cite{GordonSmithRepresentationsSRA}. Our results provide a different proof of this result.

\subsection{Contents}
We outline the structure of the paper. In Section \ref{Preliminaries} we recall definitions and lemmas we will need in later sections. In Section \ref{Results} we prove our main results for infinitesimal deformations. In Section \ref{GradedCase} we recall definitions and preliminaries in the $\mathbb{C}^*$-action case, and then in Section \ref{Graded Results} we prove our main results in this case. Finally Section \ref{Applications} contains several applications of these results, and we calculate several new examples.

\subsection{Acknowledgments}
The author is student at the University of Edinburgh, funded via the Engineering and Physical Sciences Research Council doctoral training grant number [EP/J500410/1]. The author would like to express his thanks to his supervisors, Dr.\ Michael Wemyss and Prof.\ Iain Gordon, for much guidance and patience, and also to the EPSRC.

\section{Preliminaries} \label{Preliminaries}

In this section we recall a number of definitions and theorems that we will use later, in particular relating to derived categories, tilting bundles, and deformations.

\subsection{Geometric and notational preliminaries}
Throughout this paper all schemes will be over $\mathbb{C}$, and we will often work in the generality of schemes, $X$, arising from projective morphisms $\pi:X \rightarrow \Spec(R)$ of Noetherian schemes over $\mathbb{C}$. Such schemes are Noetherian and separated, and we recall that any morphism of Noetherian schemes is quasi-compact. When $\Spec(R)$ is also of finite type then $X$ is of finite type and quasi-projective over $\mathbb{C}$, and this category of projective over affine schemes is described in \cite[Section 1.6]{BHTilting}. We let $\Coh \, X$ and $\QCoh \, X$ denote the categories of coherent and quasicoherent sheaves respectively on the scheme $X$. For an affine scheme $\Spec(R)$ we will let $\mathcal{O}_R$ denote $\mathcal{O}_{\Spec(R)}$, and make frequent similar abuses whenever it simplifies notation and  the meaning is clear. Unadorned fibre products will also always be over $\Spec(\mathbb{C})$.

\subsection{Derived categories and tilting} \label{Derived Categories}  \label{Tilting}
We recall the definitions of tilting bundles on schemes and several results related to derived categories that we will make use of later.

Consider a triangulated $\mathbb{C}$-linear category $\mathfrak{C}$ with small direct sums. A  full subcategory is \emph{thick} if it is triangulated and closed under direct summands. A full subcategory is \emph{localising} if it is triangulated and also closed under all small direct sums, and a localising subcategory is necessarily closed under direct summands \cite[Proposition 1.6.8]{Neeman Triangulated Categories}.  An object $T \in \mathfrak{C}$ \emph{generates} if the smallest localising subcategory containing $T$ is $\mathfrak{C}$. 
\begin{TiltingDef} \label{TiltingDef}
Let $\mathfrak{C}$ be a triangulated category closed under small direct sums. An object $T$ in $\mathfrak{C}$ is \emph{tilting} if:
\begin{enumerate}
\item $ \Ext^k_{\mathfrak{C}}(T,T)=0 $ for $k \neq 0$.
\item $T$ generates $\mathfrak{C}$.
\item The functor $\Hom_{\mathfrak{C}}(T,-)$ commutes with small direct sums.
\end{enumerate}
\end{TiltingDef}

For $X$ a Noetherian scheme let $D(X)=D(\QCoh \, X)$ denote the unbounded derived category of quasicoherent sheaves on $X$, let $D^b(\Coh \, X)$ denote the bounded derived category of coherent sheaves, and let $D^{-}(\Coh\, X)$ denote the bounded above derived category of coherent sheaves. For $X$ a quasi-compact and separated scheme $D(X)$ is closed under small direct sums \cite[Example 1.3]{Neeman} and $D(X)$ is compactly generated with compact objects the perfect complexes \cite[Proposition 2.5]{Neeman} which we denote by $\Perf(X)$. When $X$ is smooth the category of perfect complexes equals $D^b(\Coh \, X)$. We also note that if a quasi-compact and separated scheme $X$ has an ample line bundle $\mathcal{L}$ then the smallest localising subcategory of $D( X)$ containing $\{ \mathcal{L}^{\otimes t} : t \in \mathbb{Z} \, \}$ is $D(X)$ itself \cite[Example 1.10]{Neeman}.

For an algebra $A$ we let $D(A)$ denote the derived category of right modules over $A$ and $D^b(A$-$\mod)$ denote the bounded derived category of finitely generated right $A$-modules. When $D(X)$ has tilting object a sheaf $T$ then we define $A:=\End_{X}(T)$. When $T$ is a locally free coherent sheaf on $X$ then $T$ is a \emph{tilting bundle} and gives a derived equivalence between $D( X)$ and $D(A)$.  

\begin{Tilting}[{\cite[Theorem 7.6]{HilleVdB}}] \label{BHTilting}
Let $\pi:X \rightarrow \Spec(R)$ be a projective morphism of finite type schemes over $\mathbb{C}$ with tilting bundle $T$ on $X$ and $A:=\End_{X}(T)$. Then:
\begin{enumerate}
\item The functor $T_*:= \mathbf{R}\Hom_{X}(T,-)$ is an equivalence between $D(X)$ and $D(A)$. An inverse equivalence is given by the left adjoint $T^*:=(-) \otimes^{\mathbf{L}}_{A} T$.
\item The functors $T_*,T^*$ remain equivalences when restricted to the bounded derived categories of coherent sheaves and finitely generated modules.
\item If $X$ is smooth then $A$ has finite global dimension.
\end{enumerate}

\end{Tilting}
\begin{TiltingRemark}[{\cite[Remark 1.9]{BHTilting}}] \label{TiltingRemark} In the situation of the above theorem $A$ is a finite $R$-algebra, $A$ is module finite over its centre, and $A$ is both left and right Noetherian. Also, if there is a morphism $X \rightarrow \Spec(B)$, then the equivalence $T_*$ is $B$-linear and $A$ is naturally a $B$-algebra.
\end{TiltingRemark}

The following result gives an alternative characterisation of a tilting bundle on a projective over affine Noetherian scheme.
\begin{TodaUehara}[{\cite[Lemma 3.3]{TodaUehara}}] \label{UeharaToda} Let $\pi:X \rightarrow \Spec(R)$ be a projective morphism of Noetherian schemes, and let $T$ be a locally free coherent sheaf on $X$. Suppose that $\Ext^i_{X}(T,T)=0$ for $i \neq 0$ and $\mathbf{R}\Hom_{X}(T,\E)=0$ implies $\E \cong 0$  for any $\E \in D^{-}(\Coh \, X)$. Then $\mathbf{R}\Hom_{X}(T,-)$ and $(-)\otimes_{A}^{\mathbf{L}}T$ give inverse equivalences between $D^b(\Coh \, X)$ and $D^b(A$-$\mod)$.
\end{TodaUehara}
\begin{RemarkTU} \label{RemarkTU}
Such an equivalence between $D^b(\Coh \, X)$ and $D^b(A$-$\mod)$ produces an equivalence between compact objects, so the smallest thick subcategory containing $T$ is $\Perf(X)$. Hence $T$ automatically satisfies Definition \ref{TiltingDef} (2) and (3) as $D(X)$ is compactly generated by $\Perf(X)$ and compact objects commute with small direct sums. Hence a locally free coherent sheaf $T$  is a tilting bundle on $X$ if and only if both $\Ext_{X}^i(T,T)=0$ for $i \neq 0$ and $\mathbf{R}\Hom_X(T,\E)=0$ implies $\E \cong 0$ for any $\E \in D^{-}(\Coh \, X)$.
\end{RemarkTU}

\subsection{Derived base change} In this section we introduce notation and recall several base change results that we will require later.

\begin{DualPullback} \label{DualPullback}
Let $f:X \rightarrow Y$ be a separated morphism of Noetherian schemes over $\mathbb{C}$ and $T \in \Perf(Y)$, then 
\begin{equation*}
\mathbf{L}f^* \RCalHom_{Y}(T,\E) \cong \RCalHom_{X}(\mathbf{L}f^* T,\mathbf{L}f^* \E)
\end{equation*} 
for any $\E \in D(Y)$.
\end{DualPullback}
\begin{proof}

Consider the two functors 
\begin{align*}
\Hom_{D(X)}(\mathbf{L}f^* \RCalHom_{Y}(T,\E),-)&: D(X) \rightarrow \mathfrak{Sets}, \text{ and} \\
\Hom_{D(X)}( \RCalHom_{X}(\mathbf{L}f^*T,\mathbf{L}f^*\E),-)&: D(X) \rightarrow \mathfrak{Sets}.
\end{align*}
Showing these are naturally isomorphic proves the statement by Yoneda's lemma. This follows from the chain of natural isomorphisms
\begin{align*}
\Hom_{D(X)}(\mathbf{L}f^*\RCalHom_{Y}(T,\E), -)  & \cong \Hom_{D(Y)}(\RCalHom_{Y}(T,\E) , \mathbf{R} f_*(-)) \tag{adjunction}\\
&\cong \Hom_{D(Y)}(\E,  T \otimes^{\mathbf{L}}_{Y} \mathbf{R}f_*(-)) \tag{$T$ perfect}\\
& \cong  \Hom_{D(Y)}(\E,  \mathbf{R}f_* ( \mathbf{L}f^* T \otimes^{\mathbf{L}}_{X} (-))) \tag{projection} \\
& \cong  \Hom_{D(X)}(\mathbf{L}f^* \E,  \mathbf{L}f^* T \otimes^{\mathbf{L}}_{X} (-))  \tag{adjunction}\\
& \cong \Hom_{D(X)} (\RCalHom_{X}(\mathbf{L}f^*T,\mathbf{L}f^* \E),-).  \tag{$\mathbf{L}f^*T$ perfect}
\end{align*}
\end{proof}

We recall the following result.

\begin{Bounded Push Down} \label{Bounded Push Down} If $\pi:X \rightarrow \Spec(R)$ is a projective morphism of Noetherian schemes then $\mathbf{R}\pi_* \E \in D^b(R$-$\mod)$ for any $\E \in D^b(\Coh \,X)$. 
\end{Bounded Push Down}

\begin{proof} As $\E$ has bounded cohomology it follows that $\mathcal{H}^i(\E) = 0$ for $|i| \gg 0$ and we consider the hypercohomology spectral sequence (see \cite[Proposition 2.66]{HuybrechtsFM})
\[
E_{p,q}^2 := \mathbf{R}^p \pi_* \mathcal{H}^q(\E) \Rightarrow \mathbf{R}^{p+q} \pi_*\E.
\]
For any $q$, as $\pi$ is projective and $\mathcal{H}^q(\E)$ is a coherent sheaf it follows that $\mathbf{R}^p \pi_* \mathcal{H}^q(\E)$  is a coherent sheaf by \cite[Theorem 8.8b)]{HartshorneAG} and $\mathbf{R}^p \pi_* \mathcal{H}^q(\E)$ vanishes for $p \gg 0$ by \cite[Theorem 2.7]{HartshorneAG}. We deduce that $\mathcal{H}^{p+q}(\mathbf{R}\pi_*\E) \cong \mathbf{R}^{p+q} \pi_*\E$ is coherent for all $p,q$ and vanishes for $|p+q| \gg 0$ and hence $\mathbf{R}\pi_*\E \in D^b(R$-$\mod)$.
\end{proof}

The following lemma is due to Bridgeland, and the proof is reproduced below as we will later prove an analogue of this result in the $\mathbb{C}^*$-equivariant setting. Also note that the statement of the lemma below differs slightly from that in \cite{BridgelandEquivalences} by applying to the bounded above derived category of coherent sheaves rather than the bounded derived category of coherent sheaves, and to Noetherian schemes rather than schemes of finite type over $\mathbb{C}$. However, the same proof holds in this generality.

\begin{VitalLemma1}[{\cite[Lemma 4.3]{BridgelandEquivalences}}] \label{VitalLemma1}
Let $f:X \rightarrow Y$ be a morphism of separated Noetherian schemes over $\mathbb{C}$, and for each closed point $y \in Y$ let $j_y$ denote the inclusion of the fibre $f^{-1}(y)$. Suppose that $\mathcal{E} \in D^{-}(\Coh \, X)$ and $\mathbf{L}j_y^*\mathcal{E}$ is a sheaf for all closed points $y$. Then $\mathcal{E}$ is a coherent sheaf on $X$ flat over $Y$.
\end{VitalLemma1}
\begin{proof} 
For a closed point $y \in Y$ we consider the hypercohomology spectral sequence 
\begin{equation*}
E_2^{p,q}=\mathbf{L}_{-p}j_y^*(\mathcal{H}^{q}(\mathcal{E})) \Rightarrow \mathbf{L}_{-(p+q)}j_y^* \mathcal{E},
\end{equation*}
which is the left derived version of the sequence in \cite[Proposition 2.66]{HuybrechtsFM}.

As we assume that $\mathbf{L}j^*_y \mathcal{E}$ is a sheaf the right hand side is zero unless $p+q=0$.
As $\mathcal{E} \in D^{-}(X)$ there exists a maximal $q$ such that $\mathcal{H}^{q}(\mathcal{E})$ is nonzero, $q_0$ say. Then $\mathcal{H}^{q_0}(\mathcal{E})$ is a coherent sheaf on $X$, so is supported at some point $x$ with $f(x)=y$, hence by Nakayama's lemma $j^*_y \mathcal{H}^{q_0}(\mathcal{E}) \neq 0$ and so $E_2^{0,q_0}$ survives in the spectral sequence for this $y$. Hence $q_0=0$. 

We now consider $\mathcal{H}^0(\mathcal{E})$, and we note that this must be flat over $Y$ to prevent $E_2^{-1,0}$ surviving in the spectral sequence for some $y$.

Finally, if $\mathcal{H}^{q}(\mathcal{E} )\neq 0$ for some $q<0$ we can choose a maximal such value, $q_1$, and then $j^*_y \mathcal{H}^{q_1}(\mathcal{E})$ must be nonzero for some $y$ by Nakayama's lemma, and hence $E_2^{0,q_1}$ survives in the spectral sequence. This cannot occur, hence $\mathcal{H}^i(\mathcal{E})=0$ for $i<0$.
\end{proof}

We recall the following base change results. Let $f:X \rightarrow \Spec(D)$ and $u:\Spec(C) \rightarrow \Spec(D)$ be separated morphisms of Noetherian schemes. We consider the following base change diagram
\begin{align*}
\begin{tikzpicture} [bend angle=50]
\node (C1) at (0,0)  {$\Spec(C)$};
\node (C2) at (2,0)  {$\Spec(D)$};
\node (C3) at (0,1.6)  {$Y$};
\node (C4) at (2,1.6)  {$X$};
\draw [->] (C1) to node[above]  {\scriptsize{$u$}} (C2);
\draw [->] (C4) to node[right]  {\scriptsize{$f$}} (C2);
\draw [->] (C3) to node[left]  {\scriptsize{$g$}} (C1);
\draw [->] (C3) to node[above]  {\scriptsize{$v$}} (C4);
\end{tikzpicture}
\end{align*}
with $Y=X \times_{\Spec(D)} \Spec(C)$.

\begin{FlatBaseChange}[Flat base change {\cite[Lemma 29.5.2 (Tag 02KH)]{StacksProject}}] \label{Flat Base Change}
Suppose that $u:\Spec(C) \rightarrow \Spec(D)$ is a flat morphism of schemes. Then $u^* \mathbf{R}^i f_{*} \E \cong \mathbf{R}^if_{*}v^*\E$ for any $\E$ in $\QCoh \, X$.
\end{FlatBaseChange}

\begin{DerivedBaseChange} \label{Derived Base Change} 
Suppose that $f:X \rightarrow \Spec(D)$ is a flat morphism of schemes. Then $\mathbf{L}u^* \mathbf{R}f_{*} \E \cong \mathbf{R}f_{*} \mathbf{L}v^*\E$ for any $\E \in D(X)$.
\end{DerivedBaseChange}
\begin{proof}
As $X$ is flat over $\Spec(D)$ for any $x \in X$ and any $c \in \Spec(C)$ such that $f(x)=u(c)=d$  we have that $\Tor_i^{\mathcal{O}_{D,d}}(\mathcal{O}_{C,c},\mathcal{O}_{X,x})=0$ for all $i\neq 0$. Hence $X$ and $\Spec(C)$ are Tor independent over $\Spec(D)$, and so the result follows from \cite[Lemma 35.16.3(Tag 08ET)]{StacksProject}.
\end{proof}

The notation below is needed in order to consider base change by closed points of $\Spec(D)$. Let $\rho:X \rightarrow \Spec(D)$ be a flat morphism of Noetherian schemes such that $\rho$ factors as a projective morphism $\pi:X \rightarrow \Spec(R)$ followed by a flat affine morphism $\alpha:\Spec(R) \rightarrow \Spec(D)$.  Suppose that $u:\Spec(D/d) \rightarrow \Spec(D)$ is the inclusion of a closed point $d$ of $\Spec(D)$, and consider the pullback diagram
\[
\begin{array}{c}
 \tag{$BC_d$}\label{ClosedBaseChange}
\begin{tikzpicture} [bend angle=50]
\node (C1) at (0,0)  {$\Spec(D/d)$};
\node (C2) at (2.4,0)  {$\Spec(D)$};
\node (C3) at (0,1.6)  {$\Spec(R_d)$};
\node (C4) at (2.4,1.6)  {$\Spec(R)$};
\node (C5) at (0,3.2)  {$X_d$};
\node (C6) at (2.4,3.2)  {$X$};
\draw [->] (C1) to node[above]  {\scriptsize{$i_d$}} (C2);
\draw [->] (C4) to node[right]  {\scriptsize{$\alpha$}} (C2);
\draw [->] (C3) to node[left]  {\scriptsize{$a$}} (C1);
\draw [->] (C3) to node[above]  {\scriptsize{$l_d$}} (C4);
\draw [->] (C5) to node[above]  {\scriptsize{$j_d$}} (C6);
\draw [->] (C6) to node[right]  {\scriptsize{$\pi$}} (C4);
\draw [->] (C5) to node[left]  {\scriptsize{$p$}} (C3);
\draw [->, bend left] (C6) to node[right] {\scriptsize{$\rho= \alpha \circ \pi$}} (C2);
\draw [->, bend right] (C5) to node[left] {\scriptsize{$q = a \circ p$}} (C1);
\end{tikzpicture}
\end{array}
\]
which is referred to later as \eqref{ClosedBaseChange}.

The following corollary is deduced from the results above, and we will later use this result several times.

\begin{BaseChangeCorollary}  \label{BaseChangeCorollary} Suppose that $\E \in D^b(\Coh \, X)$ and $\mathbf{R}p_* \mathbf{L}j_d^* \E$ is a coherent sheaf on $\Spec(R_d)$ for all closed points $d \in \Spec(D)$ with diagrams \eqref{ClosedBaseChange}. Then $\mathbf{R}\pi_{*} \E$ is a coherent sheaf on $\Spec(R)$ flat over $\Spec(D)$.
\end{BaseChangeCorollary}

\begin{proof}
As $\pi$ is projective and $\E \in D^b(\Coh \, X)$ it follows that $\mathbf{R}\pi_{*} \E \in D^b(R$-$\mod)$ by Lemma \ref{Bounded Push Down}. It then follows from Lemma \ref{VitalLemma1} that if $\mathbf{L}l_d^*  \mathbf{R}\pi_{*} \E$ is a sheaf on $\Spec(R_d)$ for all closed points $d \in \Spec(D)$ then $\mathbf{R}\pi_{*} \E$ is a coherent sheaf on $\Spec(R)$ which is flat over $\Spec(D)$.

Hence we calculate $\mathbf{L}l_d^* \mathbf{R}\pi_{*} \E$. We note that as $a$ is affine $\mathbf{L}l_d^* \mathbf{R}\pi_{*} \E \in D^b(R$-$\mod)$ is a sheaf on $\Spec(R_d)$ if and only if $a_* \mathbf{L}l_d^* \mathbf{R}\pi_{*} \E$ is a sheaf on $\Spec(D/d)$. Then 
\begin{align*}
a_* \mathbf{L}l_d^* \mathbf{R}\pi_{*} \E & \cong \mathbf{L}i_d^* \alpha_{*} \mathbf{R}\pi_{*} \E \tag{Lemma \ref{Derived Base Change}}\\
&  \cong \mathbf{L}i_d^*  \mathbf{R}\rho_{*} \E  \\
&  \cong  \mathbf{R}q_{*} \mathbf{L}j_d^*   \E  \tag{Lemma \ref{Derived Base Change}} \\
& \cong  a_* \mathbf{R}p_{*} \mathbf{L}j_d^*   \E,
\end{align*}
which is a sheaf on $\Spec(D/d)$ since $\mathbf{R}p_{*} \mathbf{L}j_d^*   \E$ is a sheaf on $\Spec ( R_d)$ by hypothesis.
\end{proof}

\subsection{Definitions of deformations for schemes and algebras} \label{Deformations of Schemes}  \label{Deformations of Algebras}

Here we set the notation for deformations of a scheme.
\begin{Scheme Deformations} \hfil
\begin{enumerate}
\item
A \emph{family of schemes} is a flat morphism $\rho:X \rightarrow Y$ of schemes. The scheme $X$ is the \emph{total space} and $Y$ is the \emph{base space}.

\item
A \emph{pointed scheme}, $(Y,y)$, is a scheme $Y$ with a chosen closed point $y \in Y$. A morphism of pointed schemes $f:(Y,y) \rightarrow (Y',y')$ is a morphism $f:Y \rightarrow Y'$ which maps $y$ to $y'$.

\item
A \emph{deformation} of a scheme $X_0$ is a family of schemes $\rho:X \rightarrow (Y,y)$ with a pointed base space such that there is a chosen morphism $j_y:X_0 \rightarrow X$ inducing an isomorphism between $X_0$ and the fibre of $\rho$ over $y$. That is, the following diagram is a pullback square
\[
\begin{tikzpicture} [bend angle=0, looseness=1]
\node (C1) at (0,0)  {$y$};
\node (C2) at (0,1.6)  {$X_0$};
\node (C3) at (2,0)  {$Y$};
\node (C4) at (2,1.6)  {$X$};

\draw [->,bend right=0] (C2) to node[left]  {} (C1);
\draw [->,bend left=0] (C4) to node[right] {\scriptsize{$\rho$}} (C3);
\draw [right hook->, bend left=0] (C1) to node[above]  {} (C3);
\draw [right hook->, bend left=0] (C2) to node[above]  {\scriptsize{$j_y$}} (C4);
\end{tikzpicture}
\]
and $j_y \times_Y y:X_0 \rightarrow \rho^{-1}(y)$ is an isomorphism. A deformation is defined by the data $(\rho:X \rightarrow (Y,y),j_y:X_0 \rightarrow X)$, and the morphism $j_y$ allows the identification of $X_0$ with the fibre of $X$ over $y$. 

\item
A \emph{morphism} between deformations of $X_0$ is a map \[(f,g):(\rho:X \rightarrow (Y,y),j_y:X_0 \rightarrow X) \rightarrow (\rho':X' \rightarrow (Y',y'),j_{y'}:X_0 \rightarrow X')\] defined by $f:(Y,y)\rightarrow (Y',y')$ a morphism of pointed schemes and $g:X \rightarrow X'$ a morphism of schemes such that the following diagram is a pullback square: 

\begin{center}
$
\begin{tikzpicture} [bend angle=0, looseness=1]
\node (C1) at (0,0)  {$(Y,y)$};
\node (C2) at (0,1.6)  {$X$};
\node (C3) at (2,0)  {$(Y',y')$};
\node (C4) at (2,1.6)  {$X'$};

\draw [->,bend right=0] (C2) to node[left]  {\scriptsize{$\rho$}} (C1);
\draw [->,bend left=0] (C4) to node[right] {\scriptsize{$\rho'$}} (C3);
\draw [->, bend left=0] (C1) to node[above]  {\scriptsize{$f$}} (C3);
\draw [->, bend left=0] (C2) to node[above]  {\scriptsize{$g$}} (C4);
\end{tikzpicture}
$
\end{center}
and $g$ commutes with the inclusions of $X_0$ by $j_y$ and $j_{y'}$.
\end{enumerate}
\end{Scheme Deformations}

There is a similar notion for algebras.

\begin{Algebra Deformations} \hfil
 \begin{enumerate}
\item
 A \emph{flat family of algebras} is a map of algebras $\theta :D \rightarrow A$ such that $D$ is a commutative $\mathbb{C}$-algebra, $A$ is a $D$-algebra and $A$ is flat as a $D$-module.

\item
A \emph{deformation} of an algebra $A_0$ is a flat family of algebras $\theta :D \rightarrow A$ such that there is a chosen maximal ideal $d$ of $D$ and chosen map $u_d: A \rightarrow A_0$ such that  $u_d \otimes_{D} D/d: A\otimes_{D} {D/d} \rightarrow A_0$ is an isomorphism. A deformation is defined by the data $(\theta:D \rightarrow A,u_d:A \rightarrow A_0)$. 

\item
A morphism between deformations of $A_0$ is a map \[(f,g):(\theta: D \rightarrow A,u_d:A\rightarrow A_0) \rightarrow (\theta': D' \rightarrow A',u_{d'}:A'\rightarrow A_0)\] defined by a morphism of algebras $f:D \rightarrow D'$ such that $f^*(d)=d'$ and a morphism of algebras $g:A \rightarrow B$ such that the following diagram is a pushout square
\begin{center}
$
\begin{tikzpicture} [bend angle=0, looseness=1]
\node (C1) at (0,0)  {$D$};
\node (C2) at (0,2)  {$A$};
\node (C3) at (2,0)  {$D'$};
\node (C4) at (2,2)  {$A'$};

\draw [->,bend right=0] (C1) to node[left]  {\scriptsize{$\theta'$}} (C2);
\draw [->,bend left=0] (C3) to node[right] {\scriptsize{$\theta$}} (C4);
\draw [->, bend left=0] (C1) to node[above]  {\scriptsize{$f$}} (C3);
\draw [->, bend left=0] (C2) to node[above]  {\scriptsize{$g$}} (C4);
\end{tikzpicture}
$
\end{center}
and $f$ commutes with the maps $u_d$ and $u_{d'}$.
\end{enumerate}
\end{Algebra Deformations}
Later, the main objective we be to study deformations over local Artinian or complete local Noetherian affine schemes, with chosen closed point corresponding to the unique maximal ideal.

\section{Infinitesimal deformations and tilting} \label{Results}  \label{Lifting Tilting Bundles} In this section we consider a deformation of a scheme over a complete local Noetherian ring and prove that a tilting bundle on the scheme lifts to a tilting bundle on the deformation. This produces a deformation of the endomorphism algebra defined by the tilting bundle. 

\begin{LiftBundle} Recall that a \emph{lift} of a coherent sheaf $\mathcal{E}_0$ on a scheme $X_0$ to a deformation of that scheme $\left(\rho:X \rightarrow \Spec \, D,j_d:X_0 \rightarrow X \right)$ is a coherent sheaf $\E$ on $X$ which is flat over $\Spec(D)$ together with an isomorphism of sheaves $j_d^* \E \rightarrow \E_0$. If $\E_0$ is locally free then $\E$ is also locally free (see \cite[Exercise 7.1]{HartDef}).
\end{LiftBundle}

The following notation is used in the two lemmas and theorem below. Let $p:X_0 \rightarrow \Spec(R_0)$ be a projective morphism of Noetherian schemes over $\mathbb{C}$, suppose that the data $\left(\rho:X \rightarrow \Spec \, D, j_d:X_0 \rightarrow X \right)$ defines a deformation of $X_0$ over a complete local Noetherian ring $(D,d)$, and assume that $\rho$ factors through a projective morphism $\pi:X \rightarrow \Spec(R)$ followed by a flat affine morphism $\alpha:\Spec(R) \rightarrow \Spec(D)$ such that $R \otimes_D D/d \cong R_0$.

\begin{TiltingDeform} \label{TiltingDeform}
If $\mathcal{E}_0$ is a locally free coherent sheaf on $X_0$ such that $\Ext^i_{X_0}(\mathcal{E}_0,\mathcal{E}_0)=0$ for $i \neq 0$ then $\mathcal{E}_0$  lifts uniquely to a locally free coherent sheaf $\mathcal{E}$ on $X$ such that $\Ext^i_{X}(\E,\E)=0$ for $i \neq 0$. Moreover $A:=\End_{X}(\E)$ is flat over $D$.
\end{TiltingDeform}
\begin{proof}

We first prove that if a lift exists then it satisfies the vanishing and flatness properties. If a lift $\E$ of $\E_0$ exists then by Lemma \ref{DualPullback} $\mathbf{L} j_d^* \RCalHom_{X}(\E,\E) \cong \RCalHom_{X_0}(\mathbf{L} j_d^*\E,\mathbf{L} j_d^*\E)$, hence as $\mathbf{R}p_*\RCalHom_{X_0}(\mathbf{L}j_d^*\E,\mathbf{L}j_d^*\E) \cong \mathbf{R}\Hom_{X_0}(\E_0,\E_0)$ is a coherent sheaf on $\Spec(R_0)$  it follows by Corollary \ref{BaseChangeCorollary} that $\mathbf{R}\Hom_{X}(\E,\E)$ is a coherent sheaf on $\Spec(R)$ and is flat over $\Spec(D)$. In particular both $\Ext^i_{X}(\mathcal{E},\mathcal{E})=0$ for $i \neq 0$ and $A=\End_{X}(\mathcal{E})$ is flat as a $D$-module.

We then check that a lift of $\E_0$ does exist. We reproduce an argument here similar to the proof of \cite[Theorem 2.10]{SympMcKay}. Define the family of local Artinian $\mathbb{C}$-algebras $(D_{n},d_{n}):= (D/d^{n+1},d/d^{n+1})$ and the inverse limit of this family is $(D,d)$. We define $R_{n} = R \otimes_D D_n$ and $R$ is the inverse limit of the $(R_n)$, and similarly we define $X_{n} :=X \times_{\Spec(D)} \Spec(D_n)$, which are a family of deformations $q_n:X_{n}\rightarrow \Spec(D_n)$ of $X_0$ over the Artinian, graded local rings $D_n$ such that each $X_{n}$ is projective over $\Spec(R_{n})$. We show by induction that there is a unique lift of $\E_0$ to each $X_i$. Suppose $\E_i$ is a lift of $\E_0$ to $X_i$. By \cite[Theorem 7.1]{HartDef} if $\CH^2(X_i,\CalEnd_{X_i}(\mathcal{E}_i)\otimes_{X_i} q_i^*d_i)=0$ a lift of $\E_i$ to $X_{i+1}$ exists and if $\CH^1(X_i,\CalEnd_{X_i}(\mathcal{E}_i)\otimes_{X_i} q_i^*d_i)=0$ it is unique. But by the previous paragraph $\E_i$ satisfies $\Ext^j_{X_i}(\mathcal{E}_i,\mathcal{E}_i)=0$ for $j \neq 0$, and so we can deduce  $\CH^j(X_i,\CalEnd_{X_i}(\mathcal{E}_i)\otimes_{X_i} q_i^*d_i)\cong \mathbf{R}^j q_{i*}(\CalEnd_{X_i}(\mathcal{E}_i)\otimes_{X_i} q_i^*d_i) \cong  \Ext^j_{X_i}(\mathcal{E}_i,\mathcal{E}_i)\otimes_{\mathbb{C}} d_i =0 $ for $j \neq 0$ and a unique lift $\E_{i+1}$ on $X_{i+1}$ exists. By induction $\E_0$ lifts uniquely to a locally free coherent sheave $\E_i$ on each $X_i$, and so by \cite[Remark 29.22.8 (Tag 088F)]{StacksProject} it lifts uniquely to a locally free coherent sheaf $\E$ on $X$. 

\end{proof}

\begin{BetterGeneration} \label{BetterGeneration} Suppose that $T_0$ is a tilting bundle on $X_0$ and $T$ is the lift of  $T_0$ to $X$, then $T$ is a tilting bundle on $X$.
\end{BetterGeneration}
\begin{proof}
Firstly, $T$  satisfies Definition \ref{TiltingDef}(1) as $\Ext^i_{X}(T,T)=0$ for $i\neq 0$ by Lemma \ref{TiltingDeform}, and $T$ also satisfies Definition \ref{TiltingDef}(3) as $T$ is a coherent locally free sheaf hence a compact object of $D(X)$. Hence to prove that $T$ is a tilting bundle we need only show that $T$ generates, Definition \ref{TiltingDef}(2), and to do this we check the condition outlined in  Remark \ref{RemarkTU}: that $\mathbf{R}\Hom_{X}(T,\mathcal{F})\cong 0$ implies $\mathcal{F}\cong 0$ for $\mathcal{F} \in D^{-}(\Coh \, X)$. Suppose that $\E \in D^{-}(\Coh \, X)$ and $\mathbf{R}\Hom_{X}(T,\mathcal{F}) \cong 0$ and consider the base change diagram (\ref{ClosedBaseChange}) for the inclusion of the unique closed point $d \in \Spec(D)$. Then 
\begin{align*}
0 & \cong \mathbf{L} i_d^* \mathbf{R}\Hom_{X}(T,\mathcal{F}) \\
& \cong \mathbf{L} i_d^* \mathbf{R} \rho_{*} \RCalHom_{X}(T,\mathcal{F}) \\
& \cong  \mathbf{R} q_* \mathbf{L} j_d^* \RCalHom_{X}(T,\mathcal{F}) \tag{Lemma \ref{Derived Base Change}} \\
& \cong \mathbf{R}\Hom_{X_0}(\mathbf{L}j_d^*T, \mathbf{L}j_d^*\mathcal{F}). \tag{Lemma \ref{DualPullback}}
\end{align*}
As $\mathbf{L}j_d^*T \cong T_0$ is a tilting bundle on $X_0$ and $\mathbf{R}\Hom_{X_0}(T_0,\mathbf{L}j_d^*\mathcal{F})=0$ it follows that $\mathbf{L}j_d^* \mathcal{F}=0$. We deduce that $\mathcal{F}$ is a coherent sheaf on $X$ which is flat over $\Spec(D)$ by Lemma \ref{VitalLemma1}, and hence $j_d^*\mathcal{F}=0$ so $\mathcal{F}=0$ by Nakayama's lemma.
\end{proof}

The following theorem is the main result of this section.

\begin{Main} \label{Main} 
Any tilting bundle $T_0$ on $X_0$ lifts uniquely to a tilting bundle $T$ on $X$. If $A_0=\End_{X_0}(T_0)$ then  $A:=\End_{X}(T)$ is a $D$-algebra which is flat as a $D$-module, there is a map $A \rightarrow A_0$ such that $A \otimes_{D} D/d \cong A_0$, and this map and isomorphism defines $D \rightarrow A$ as a deformation of $A_0$. 
\end{Main}
\begin{proof}
Let $T_0$ be a tilting bundle on $X_0$. By Lemmas \ref{TiltingDeform} and \ref{BetterGeneration} $T_0$ lifts uniquely to a tilting bundle $T$ on $X$ and $A=\End_{X}(T)$ is a flat module over $D$. In particular the flat morphism $\rho:X \rightarrow \Spec(D)$ produces a map $\theta:D \rightarrow A=\End_{X}(T)$ defining $A$ as a flat $D$-algebra. The map $j_d:X_0 \rightarrow X$ defines a unit map $T \rightarrow j_{d*}j_d^*T$ and hence as $j_d^*T\cong T_0$ applying $\Hom_X(-,T)$ gives a morphism $u_d :A= \Hom_{X}(T,T) \rightarrow \Hom_{X}(T,j_{d*}j_d^* T) \cong \Hom_{X_0}(j_d^*T,j_d^*T) \cong A_0$. Then 
\begin{align*}
A \otimes_{D}D/d & := \mathbf{L}i_d^*\mathbf{R}\rho_{*} \RCalHom_{X}(T,T) \tag{$A$ is flat over $D$} \\
&=\mathbf{R}q_{*} \mathbf{L}j_d^* \RCalHom_{X}(T,T) \tag{Lemma \ref{Derived Base Change}}\\
&  \cong  \mathbf{R}\Hom_{X_0}(\mathbf{L}j_d^* T, \mathbf{L}j_d^*T)  \tag{Lemma \ref{DualPullback}}\\
 &\cong \mathbf{R}\Hom_{X_0}(T_0,T_0)=A_0 \tag{$j_d^*T \cong T_0$}.
\end{align*}
This isomorphism $A \otimes_D D/d \cong A_0$ is uniquely defined by the identification of $X_0$ with the fibre of $X$ over $d$ and the isomorphism $j_d^*T \cong T_0$. Hence $(\theta:D \rightarrow A, u_d:A \rightarrow A_0$) is a deformation of $A_0$ determined by the deformation of $X_0$.
\end{proof}

\section{$\mathbb{C}^*$-Equivariant case} \label{GradedCase}
The results of the previous section rely on the assumption that the deformation is over a scheme with a unique closed point. We consider an extension which maintains this property in a more global setting: working with a deformations over a base space with a good $\mathbb{C}^*$-action. Such schemes are not necessarily local but have a unique $\mathbb{C}^*$-fixed closed point which occurs in the closure of every orbit, and we extend our results to this setting.

\subsection{Basics on graded rings and $\mathbb{C}^*$-actions} 
In this section we recall the basic notions, and set notation.

\begin{GradedDefinitions} \hfil
\begin{enumerate}
\item
A \emph{graded $\mathbb{C}$-algebra} is a $\mathbb{C}$-algebra $R$ with a decomposition $R=\bigoplus_{i \in \mathbb{Z}} R_i$ as a $\mathbb{C}$-vector space such that $R_iR_j \subset R_{i+j}$ for all $i,j \in \mathbb{Z}$. It is \emph{positively graded} if $R_i=0$ for $i <0$.

\item
A \emph{graded $R$-module} is an $R$-module with a decomposition $M=\bigoplus_{i \in \mathbb{C}} M_i$ as a $\mathbb{C}$-vector space such that $R_iM_j \subset M_{i+j}$ for all $i,j \in \mathbb{Z}$.

\item
A \emph{graded ideal} is an ideal which is also a graded module such that inclusion morphism $\sigma$ satisfies $\sigma(N_i) \subset M_i$ for all $i$. 

\item
A \emph{graded maximal ideal} is a graded ideal which is maximal under inclusion among proper graded ideals. 

\item
A graded algebra is \emph{graded local} if it has a unique maximal graded ideal. If $R$ is a positively graded $\mathbb{C}$-algebra and $\frak{m}$ is a graded maximal ideal then $R/\frak{m} \cong \mathbb{C}$.

\item
We note that inverse limits exist in the category of positively graded rings. A positively graded algebra $R$ is \emph{graded complete} with respect to a graded maximal ideal $\mathfrak{m}$ if the projective limit of $(R/ \mathfrak{m}^n)_{n \ge 0}$ in the category of positively graded rings is isomorphic to $R$. 
\end{enumerate}
\end{GradedDefinitions}

\begin{GradedRingRemarks}
Graded local rings are not necessarily local rings, and rings which are graded complete with respect to a graded ideal are not necessarily complete with respect to that ideal. For example, the graded ring $\mathbb{C}[x]$ with $x$ in degree 1 has unique graded maximal ideal $(x)$ and is graded complete with respect to this ideal however is neither a local ring nor complete with respect to $(x)$ as an ungraded ring.

The graded completion of a positively graded ring with respect to a maximal ideal graded in strictly positive degree can be recovered from the (ungraded) completion of the ring by taking the vector subspace spanned by eigenvectors of the corresponding $\mathbb{C}^*$-action, see \cite[Lemma A2]{NamikawaFlopsandPoisson}.
\end{GradedRingRemarks}

We recall the definition of a $\mathbb{C}^*$-action on a scheme.

\begin{C*Action} \label{C* Action} \hfil
\begin{enumerate}
\item
A $\mathbb{C}^*$-action on a finite type scheme $X$ is defined by an \emph{action map}
\begin{align*}
\sigma: \mathbb{C}^* \times X  & \rightarrow X 
\end{align*}
such that the following diagrams commute
\begin{align*}
\begin{tikzpicture} [bend angle=0]
\node (C1) at (0,0)  {$\mathbb{C}^* \times X$};
\node (C2) at (3,0)  {$X$};
\node (C3) at (0,1.6)  {$\mathbb{C}^* \times \mathbb{C}^* \times X$};
\node (C4) at (3,1.6)  {$\mathbb{C}^* \times X$};
\draw [->] (C1) to node[above]  {\scriptsize{$\sigma$}} (C2);
\draw [->] (C4) to node[right]  {\scriptsize{$\sigma$}} (C2);
\draw [->] (C3) to node[left]  {\scriptsize{$id_{\mathbb{C}^*} \times\sigma$}} (C1);
\draw [->] (C3) to node[above]  {\scriptsize{$m \times id_{X}$}} (C4);
\end{tikzpicture}
\qquad
\begin{tikzpicture} [bend angle=0]
\node (B1) at (0,0)  {$X$};
\node (B2) at (0,1.6)  {$\mathbb{C}^* \times X$};
\node (B3) at (2.2,1.6)  {$X$};
\draw [->] (B1) to node[below]  {\scriptsize{$id_X$}} (B3);
\draw [->] (B1) to node[left]  {\scriptsize{$1 \times id_X$}} (B2);
\draw [->] (B2) to node[above]  {\scriptsize{$\sigma$}} (B3);
\end{tikzpicture}
\end{align*}
where $m$ is the multiplication map, $m: \mathbb{C}^* \times \mathbb{C}^* \rightarrow \mathbb{C}^*$, and we also have the projection map $\beta:\mathbb{C}^* \times X \rightarrow X$.

\item
A morphism $f:X \rightarrow Y$ between two schemes with $\mathbb{C}^*$-actions is $\mathbb{C}^*$-\emph{equivariant} if the following diagram commutes
\begin{align*}
\begin{tikzpicture} [bend angle=0]
\node (C1) at (0,0)  {$X$};
\node (C2) at (3,0)  {$Y$};
\node (C3) at (0,1.6)  {$\mathbb{C}^* \times X$};
\node (C4) at (3,1.6)  {$\mathbb{C}^* \times Y$};
\draw [->] (C1) to node[above]  {\scriptsize{$f$}} (C2);
\draw [->] (C4) to node[right]  {\scriptsize{$\sigma$}} (C2);
\draw [->] (C3) to node[left]  {\scriptsize{$\sigma$}} (C1);
\draw [->] (C3) to node[above]  {\scriptsize{$(id \times f)$}} (C4);
\end{tikzpicture}
\end{align*}

\item
If $X$ is a scheme with $\mathbb{C}^*$-action a quasicoherent sheaf $\mathcal{F}$ on $X$ is $\mathbb{C}^*$-\emph{equivariant} if it is equipped with an isomorphism
\begin{equation*}
\phi: \sigma^* \mathcal{F} \xrightarrow{\sim} \beta^* \mathcal{F}
\end{equation*}
such that the pullback $(1 \times id_X)^*\phi:\mathcal{F} \rightarrow\mathcal{F}$ is the identity and the following diagram commutes in the category of coherent sheaves on $\mathbb{C}^* \times \mathbb{C}^* \times X$
\begin{align*}
\begin{tikzpicture} [bend angle=0]
\node (C1) at (0,0)  {$(id_{\mathbb{C}^*} \times \sigma)^*a^*\mathcal{F}$};
\node (C2) at (3,0)  {$(m\times id_X)^* \sigma^* \mathcal{F}$};
\node (C3) at (0,1.6)  {$(id_{\mathbb{C}^*} \times \sigma)^*p^*\mathcal{F}$};
\node (C4) at (3,1.6)  {$\beta^*\mathcal{F}$};
\draw [->] (C1) to node[above]  {\scriptsize{$\cong$}} (C2);
\draw [->] (C2) to node[right]  {\scriptsize{$(m \times id_X)^*\phi$}} (C4);
\draw [->] (C1) to node[left]  {\scriptsize{$(id_{\mathbb{C}^*} \times \sigma)^*\phi$}} (C3);
\draw [->] (C3) to node[above]  {\scriptsize{$\beta_{23}^*\phi$}} (C4);
\end{tikzpicture}
\end{align*}
where $\beta_{23}$ is projection onto the second and third components of $\mathbb{C}^* \times \mathbb{C}^* \times X$. When we refer to a $\mathbb{C}^*$-equivariant sheaf we will often suppress the chosen isomorphism $\phi$.
\end{enumerate}
\end{C*Action}
 We denote the abelian category of $\mathbb{C}^*$-equivariant quasicoherent sheaves on $X$ by $\QCoh^{\mathbb{C}^*} X$, where the objects are $\mathbb{C}^*$-equivariant quasicoherent sheaves on $X$ and the morphisms in this category between two objects $\mathcal{E},\mathcal{F}$ are the $\mathbb{C}^*$-invariant morphisms in $\Hom_X(\E,\mathcal{F})$, which we denote by $\Hom^{\mathbb{C}^*}_X(\E,\mathcal{F})$. We let $\Coh^{\mathbb{C}^*} X$ denote the full abelian subcategory of $\mathbb{C}^*$-equivariant coherent sheaves. We also recall that there are enough injectives in $\QCoh^{\mathbb{C}^*} X$, \cite[Corollary 1.5.6]{VdBEquivariant}, hence the unbounded derived category of $\mathbb{C}^*$-equivariant quasicoherent sheaves exists and we denote it by $D_{\mathbb{C}^*}(X)$.

The following lemmas recalls properties of $\mathbb{C}^*$-equivariant coherent sheaves.
\begin{ActionLemma}[See e.g. {\cite[Section 3.3]{BlumeThesis}}] \label{Action Lemma} If $X$ and $Y$ are Noetherian schemes with $\mathbb{C}^*$-actions and $\pi:X \rightarrow Y$ is a $\mathbb{C}^*$-equivariant separated morphism then:
\begin{enumerate}
\item If $\mathcal{F} \in \Coh^{\mathbb{C}^*}  X$ is locally free and $\mathcal{G} \in \QCoh^{\mathbb{C}^*}X$ then $\SHom_{X}(\mathcal{F},\mathcal{G}) \in \QCoh^{\mathbb{C}^*}  X$.
\item If $\mathcal{F}, \mathcal{G} \in \QCoh^{\mathbb{C}^*}  X$ then $\mathcal{F} \otimes_X \mathcal{G} \in \QCoh^{\mathbb{C}^*}  X$.
\item If $\mathcal{F} \in \QCoh^{\mathbb{C}^*} \, X$ and $\mathcal{G} \in \QCoh^{\mathbb{C}^*} \, Y$ then $\pi_* \mathcal{F} \in \QCoh^{\mathbb{C}^*} Y$ and $\pi^* \mathcal{G} \in \QCoh^{\mathbb{C}^*}  X$.
\end{enumerate}
\end{ActionLemma}
There is a well known relations between graded rings and $\mathbb{C}^*$-actions on finite type schemes.
\begin{GradingActionRemark}[{See e.g. \cite[Example 3.4]{HashimotoEquivariant}}] \label{GradingActionRemark}
If $R$ is a commutative $\mathbb{C}$-algebra of finite type then:
\begin{enumerate}
\item A $\mathbb{Z}$-grading on $R$ is equivalent to a $\mathbb{C}^*$-action on the corresponding affine scheme $\Spec(R)$. 
\item If $R$ is graded the category of graded $R$-modules is equivalent to the category of $\mathbb{C}^*$-equivariant coherent sheaves on $\Spec(R)$.
\end{enumerate}
\end{GradingActionRemark}

 A $\mathbb{C}^*$-action on an affine scheme $\Spec(R)$ of finite type over $\mathbb{C}$ is defined by a homomorphism $\mathbb{C}^* \rightarrow \Aut_{\mathbb{C}}(R)$  induced from a $\mathbb{C}$-algebra morphism $R \rightarrow R \otimes_{\mathbb{C}} \mathbb{C}[t, t^{-1}]$. We also recall the definition of a $\mathbb{C}^*$-action on a scheme which is not of finite type but which occurs as the completion of a finite type scheme.
\begin{CompleteActionDefinition}When $\Spec(R)$ is not of finite type but $(R,\mathfrak{m})$ is a Noetherian complete local $\mathbb{C}$-algebra with maximal ideal $\frak{m}$ then a $\mathbb{C}^*$-action on $\Spec(R)$ is defined by a homomorphism $\mathbb{C}^* \rightarrow \Aut_{\mathbb{C}}(R)$ induced from a $\mathbb{C}$-algebra morphism  $R \rightarrow R \widehat{\otimes}_{\mathbb{C}}\mathbb{C}[t, t^{-1}]$, where  $R \widehat{\otimes}_{\mathbb{C}}\mathbb{C}[t, t^{-1}]$ is the completion of  $R \otimes_{\mathbb{C}} \mathbb{C}[t, t^{-1}]$ in the maximal ideal $\frak{m}(R \otimes_{\mathbb{C}}\mathbb{C}[t,t^{-1}])$ \cite[A1]{NamikawaFlopsandPoisson}. We define $\Spec(R) \hat{\times} \mathbb{C}^*:= \Spec(R \widehat{\otimes}_{\mathbb{C}}\mathbb{C}[t, t^{-1}])$, and then this definition of $\mathbb{C}^*$-action matches that of Definition \ref{C* Action} but with completed fibre products.

Similarly, if $X \rightarrow \Spec(R)$ is a projective morphism such that $\Spec(R)$ is complete local then a $\mathbb{C}^*$-action on $X$ is defined to as in Definition \ref{C* Action} but replacing all fibre produces with completed fibre products $X \hat{\times}\mathbb{C}^*$, where the completed fibre product $X \hat{\times} \mathbb{C}^*$ is defined as the fibre product of $X\times \mathbb{C}^* \rightarrow \Spec(R) \times \mathbb{C}^*$  and $\Spec(R) \hat{\times} \mathbb{C}^*  \rightarrow \Spec(R) \times \mathbb{C}^*$.
\end{CompleteActionDefinition}

\begin{GoodAction} A $\mathbb{C}^*$-action on an affine scheme $\Spec(R)$ is \emph{good} if there is a unique $\mathbb{C}^*$-fixed closed point corresponding to a maximal ideal $\mathfrak{m}$, such that $\mathbb{C}^*$ only acts by strictly positive weights on $\mathfrak{m}$. If $\Spec(R)$ has a good $\mathbb{C}^*$-action, then the closure of any non-empty $\mathbb{C}^*$ orbit on $\Spec(R)$ must contain the unique $\mathbb{C}^*$-fixed closed point.
\end{GoodAction}

\subsection{$\mathbb{C}^*$-Equivariant tilting}

Suppose that $\pi:X \rightarrow \Spec(R)$ is a $\mathbb{C}^*$-equivariant projective morphism of finite type schemes with $\mathbb{C}^*$-actions, $\mathcal{E} \in \Coh^{\mathbb{C}^*}X$, and $\mathcal{F} \in \QCoh^{\mathbb{C}^*}X$. Then $R$ is a graded ring and $\Hom^{\mathbb{C}^*}_X(\E,\mathcal{F})$ equals the degree $0$ part of $\Hom_X(\E,\mathcal{F})=\pi_* \CalHom_X(\E,\mathcal{F})$ considered as a graded $R$-module. We recall that the characters of $\mathbb{C}^*$ are given by $ a \mapsto a^i$ for $i \in \mathbb{Z}$ and tensoring with these characters, which we note is flat, induces twisted $\mathbb{C}^*$-equivariant structures $\E(i)$ on $\E$. If we let $(i)_R$ denote the grading shift for $\Spec(R)$ then $\E(i)= \E \otimes_X \pi^* R(i)_R$ and $\pi_* (\E(i)) \cong (\pi_i \E)(i)_R$ by the projection formula. In particular, $\Hom^{\mathbb{C}^*}_X(\E,\mathcal{F}(i))$ is the degree $i$ part of $\Hom_{X}(\E,\mathcal{F})$ considered as a graded $R$-module.  By the lemma immediately below we can consider $A=\End_X(\E)$ as a graded algebra,  and this allows us to consider interactions between tilting bundles and $\mathbb{C}^*$-actions.

\begin{GradedTiltingBundles}
Let $\pi:X \rightarrow \Spec(R)$ be a $\mathbb{C}^*$-equivariant projective morphism of finite type schemes with $\mathbb{C}^*$-actions, and let $\E$ be a $\mathbb{C}^*$-equivariant locally free coherent sheaf on $X$. Then the ring $A=\End_X(\E)$ has a natural grading making it a graded $R$-module.  In particular, $A=\Hom_X(\E,\E)=\bigoplus_{i \in \mathbb{Z}} \Hom^{\mathbb{C}^*}_X(\E(-i),\E)$.
\end{GradedTiltingBundles}
\begin{proof}
The action map $\sigma:X \times \mathbb{C}^* \rightarrow X$ defines a map of rings $\phi:A=\Hom_X(\E,\E) \rightarrow A \otimes_{\mathbb{C}} \mathbb{C}[t,t^{-1}]=\Hom_{X \times \mathbb{C}^*}(\beta^*\E,\beta^*\E )$ induced by the unit map $\E \rightarrow \sigma_*\sigma^*\E$ and isomorphism $\beta^*\E \cong \sigma^*\E$ defining the $\mathbb{C}^*$-equivariant structure on $\E$. This defines an algebraic action of $\mathbb{C}^*$ on $A$ by $q.a$ equalling the evaluation of $\phi(a)$ at $t=q$ for $a \in A$. Hence $A$ decomposes into eigenspaces for this action which define $A$ as a $\mathbb{Z}$-graded ring. As $\E$ is $\mathbb{C}^*$-equivariant it follows from Lemma \ref{Action Lemma}(1) that $\CalEnd_{X}(\E)$ is $\mathbb{C}^*$-equivariant, and then it follows by Lemma \ref{Action Lemma}(3) and Remark \ref{GradingActionRemark}(2) that $\End_{X}(\E)$ is a graded $R$-module.
\end{proof}

Suppose now that $\pi:X \rightarrow \Spec(R)$ is a $\mathbb{C}^*$-equivariant projective morphism of finite type schemes with $\mathbb{C}^*$-actions and we have a $\mathbb{C}^*$-equivariant locally free coherent sheaf $T$ such that $\Ext^j_{X}(T,T)=0$ for $j \neq 0$. We wish to investigate when properties of the derived category of $\mathbb{C}^*$-equivariant quasicoherent sheaves can be used to study generation for the derived category of quasicoherent sheaves.  The functor $\Hom_{X}(T,-)$ carries a $\mathbb{C}^*$-equivariant structure on a sheaf $\E$ to a $\mathbb{C}^*$-equivariant structure on the $A$-module $\Hom_X(T,\E)$ by Lemma \ref{Action Lemma}, and this can be interpreted as the grading on the $A$-module $\Hom_X(T,\E)$ given by $\Hom_X(T,\E)= \bigoplus_{i \in \mathbb{Z}} \Hom^{\mathbb{C}^*}_X(T,\E(i))$. Then the  $\mathbb{C}^*$-equivariant bundle $T$ defines adjoint functors \[T^{\mathbb{C^*}}_*:= \bigoplus_{i \in \mathbb{Z}}\mathbf{R}\Hom_{X}(T(-i),-)^{\mathbb{C}^*}:D_{\mathbb{C}^*}(X) \rightarrow D_{\mathbb{C}^*}(A)\] and
\[ T^*_{\mathbb{C}^*}:=T \otimes ^{\mathbf{L}}_{A}(-):D_{\mathbb{C}^*}(A) \rightarrow D_{\mathbb{C}^*}(X)
\]
where $D_{\mathbb{C}^*}(A)$ is the unbounded derived category of $\mathbb{C}^*$-equivariant right $A$-modules, which correspond to graded right $A$-modules. We can also relate these functors back to the corresponding functors without the $\mathbb{C}^*$-action. Consider the forgetful functors $F_A:D_{\mathbb{C}^*}(A) \rightarrow D(A)$ and $F_X: D_{\mathbb{C}^*}(X) \rightarrow D(X)$ and the following digram commutes 
\begin{align*}
\begin{tikzpicture} [bend angle=0]
\node (C1) at (0,0)  {$D(X)$};
\node (C2) at (3,0)  {$D(A)$};
\node (C3) at (0,1.6)  {$D_{\mathbb{C}^*}(X)$};
\node (C4) at (3,1.6)  {$D_{\mathbb{C}^*}(A)$};
\node (C5) at (6,0)  {$D(X)$};
\node (C6) at (6,1.6)  {$D_{\mathbb{C}^*}(X)$};
\draw [->] (C1) to node[above]  {\scriptsize{$T_*$}} (C2);
\draw [->] (C4) to node[left]  {\scriptsize{$F_A$}} (C2);
\draw [->] (C3) to node[left]  {\scriptsize{$F_X$}} (C1);
\draw [->] (C3) to node[above]  {\scriptsize{$T^{\mathbb{C}^*}_*$}} (C4);
\draw [->] (C6) to node[left]  {\scriptsize{$F_X$}} (C5);
\draw [->] (C4) to node[above]  {\scriptsize{$T_{\mathbb{C}^*}^*$}} (C6);
\draw [->] (C2) to node[above]  {\scriptsize{$T^*$}} (C5);
\end{tikzpicture}
\end{align*}
where $T_*= \mathbf{R}\Hom_{X}(T,-):D(X) \rightarrow D(A)$ and $T^*:= (-) \otimes_{A}^{\mathbf{L}} T:D(A) \rightarrow D(X)$.
We note the following consequence of this discussion, which we will require later.
\begin{ShiftedGradingRemark} \label{ShiftedGradingRemark}
For any $\mathbb{C}^*$-equivariant quasicoherent sheaf $\E$ on $X$ the counit morphism $T^{*}_{\mathbb{C}^*} T_*^{\mathbb{C^*}}(\E) \rightarrow \E$ in $D_{\mathbb{C}^*}(X)$ restricts,  after applying the forgetful functor, to coincide with the counit morphism $T^*T_* (\E) \rightarrow \E$ in $D(X)$.
\end{ShiftedGradingRemark}

We also recall the following lemma that allows us to assume the existence of $\mathbb{C}^*$-equivariant ample line bundles on normal quasi-projective Noetherian schemes with $\mathbb{C}^*$-actions which is a fact we will use later.
\begin{Equivariant Ample Line Bundle}[{\cite[Theorem 1.6]{SumihiroEquivariantCompletetionII}}] \label{Equivariant Ample Line Bundle} If $X$ is a normal Noetherian scheme with a $\mathbb{C}^*$-action then for any line bundle $\mathcal{L}$ on $X$ there exists some positive integer $n$ such that $\mathcal{L}^{\otimes n}$ can be given a $\mathbb{C}^*$-equivariant structure.
\end{Equivariant Ample Line Bundle}
\begin{ActionGenerationRemark}
If $X$ is quasi-projective then it has an ample line bundle and this theorem allows us to assume the existence of a $\mathbb{C}^*$-equivariant ample line bundle on $X$. It is noted in Section \ref{Derived Categories} that  if a scheme $X$ has an ample line bundle $\mathcal{L}$ then the smallest localising subcategory of $D(X)$ containing $\{ \mathcal{L}^{\otimes t} : t \in \mathbb{Z} \, \}$ is $D(X)$ itself, hence for a normal quasi-projective variety $X$ with $\mathbb{C}^*$-action the smallest localising subcategory of $D(X)$ containing all coherent sheaves that can be given a $\mathbb{C}^*$-equivariant structure is $D(X)$ itself.
\end{ActionGenerationRemark}

\subsection{$\mathbb{C}^*$-Equivariant deformations} \label{GradedDeformationsandSchemes} We recall the definitions of graded deformations of algebras and $\mathbb{C}^*$-equivariant deformations of schemes. 

\begin{Graded Scheme Deformations}
A $\mathbb{C}^*$-\emph{equivariant deformation} of a scheme with $\mathbb{C}^*$-action is defined to be a deformation of a scheme as in Section \ref{Deformations of Schemes} such that there is a $\mathbb{C}^*$-action on all schemes, all morphisms are $\mathbb{C}^*$-equivariant, and the chosen closed point is invariant under the $\mathbb{C}^*$-action.
\end{Graded Scheme Deformations}

\begin{Graded Algebra Deformations}
We will also be interested in \emph{graded algebra deformations}. These are algebra deformations as in Section \ref{Deformations of Algebras}, but all algebras are required to be graded, all morphisms required to be graded, and the maximal ideal is required to be a graded maximal ideal.
\end{Graded Algebra Deformations}

A particular class of graded deformations are those over the Noetherian graded complete local ring $\mathbb{C}[t]$ where $t$ has grade 1. Taking the fibre of such an algebra when $t =1$ allows the construction of filtered algebras whose associated graded algebras are isomorphic to the algebra being deformed.

\begin{ReesRingGradedDeformations}[{\cite[Lemma 1.3 and 1.4]{BrG}}] \label{ReesRingGradedDeformations}
If $(\mathbb{C}[t] \rightarrow A, u_t:A \rightarrow A_0)$ is a graded deformation of a positively graded ring $A_0$ then $A/(t$-$1)A$ is a filtered algebra with associated graded algebra isomorphic to $A$. In this case the algebra $A$ is the Rees ring of the filtered algebra $A/(t$-$1)A$ and the associated graded algebra of $A/(t$-$1)A$ is isomorphic to $A/tA \cong A_0$.
\end{ReesRingGradedDeformations}

\begin{ReesFibresRemark} \label{ReesFibresRemark}Suppose that $(D,d)$ is a graded complete local ring with graded maximal ideal in strictly positive degree and $(D \rightarrow A,u_d:A\rightarrow A_0)$ is a deformation of a positively graded algebra $A_0$. Then there is a graded morphism $D\rightarrow \mathbb{C}[t]$ (with $t$ in degree 1) sending homogeneous generators of $D$ to powers of $t$ and hence mapping the maximal ideal $d$ to the maximal ideal $(t)$. Base changing by this map allows the application of the above result to any fibre of $A$ over $\Spec(D)$.
\end{ReesFibresRemark}

\subsection{$\mathbb{C}^*$-Equivariant preliminary lemmas}
We must rephrase several lemmas that we have used previously in the $\mathbb{C}^*$-equivariant setting, and we use the following notation for the two lemmas below. Let $f:X \rightarrow Y=\Spec(R)$ be a $\mathbb{C}^*$-equivariant morphism of  Noetherian schemes with $\mathbb{C}^*$-actions, suppose that the $\mathbb{C}^*$-action on $Y$ is good with unique fixed closed point $y$, and let $j_y:f^{-1}(y) \rightarrow X$ denote the inclusion of this fibre. 

We first recall the $\mathbb{C}^*$-equivariant version of Nakayama's lemma. 
\begin{GradedNakayama} \label{GradedNakayama} If $\E$ is a $\mathbb{C}^*$-equivariant coherent sheaf on $X$ such that $j^*_y \E \cong 0$ then $\E \cong 0$.
\end{GradedNakayama}
\begin{proof}
For any closed point $z \in Y$ let $j_z$ denote the inclusion of this fibre in $X$. If $\E$ is a coherent sheaf on $X$ then by Nakayama's lemma if $j_z^* \E=0$ for all closed points $z \in Y$ then $\E=0$. We also know that as $\E$ is a coherent sheaf the support of $\E$ is closed \cite[Lemma 17.9.6 (Tag 01BA)]{StacksProject}, as it is $\mathbb{C}^*$-equivariant we know that its set theoretic support is closed under the $\mathbb{C}^*$ action \cite[Lemma 8.1 6]{HashimotoEquivariant} hence as the $\mathbb{C}^*$-action on $\Spec(R)$ is good $\E$ must be supported over the unique $\mathbb{C}^*$-fixed closed point $y$. Therefore if $j_y^* \E \cong 0 $ then $j_z^*\E \cong 0$ for all closed points $z \in Y$, hence $\E \cong 0$.
\end{proof}

We then note the following $\mathbb{C}^*$-equivariant version of Lemma \ref{VitalLemma1}.

\begin{GradedVitalLemma1} \label{GradedVitalLemma1} Suppose that $\mathcal{E} \in D^{-}(\Coh \, X)$ has $\mathbb{C}^*$-equivariant cohomology sheaves and $\mathbf{L}j_y^*\mathcal{E}$ is a $\mathbb{C}^*$-equivariant coherent sheaf. Then $\mathcal{E}$ is a $\mathbb{C}^*$-equivariant coherent sheaf on $X$ and is flat over $\Spec(R)$.
\end{GradedVitalLemma1}
\begin{proof}
We again copy the proof of \cite[Lemma 4.3]{BridgelandEquivalences} but in the $\mathbb{C}^*$-equivariant setting.
We consider the hypercohomology spectral sequence
\begin{equation*}
E_2^{p,q}=\mathbf{L}_{-p}j_y^*(\mathcal{H}^{q}(\mathcal{E})) \Rightarrow \mathbf{L}_{-(p+q)}j_y^* \mathcal{E}
\end{equation*}
for the unique $\mathbb{C}^*$-fixed point $y \in Y$.
As we assume that $\mathbf{L}j^*_y \mathcal{E}$ is a sheaf the right hand side is zero unless $p+q=0$.
As $\mathcal{E} \in D^{-}(\Coh \, X)$ there exists a maximal $q$ such that $\mathcal{H}^{q}(\mathcal{E})$ is non zero, $q_0$ say. Then $\mathcal{H}^{q_0}(\mathcal{E})$ is a non-zero $\mathbb{C}^*$-equivariant coherent sheaf on $X$, so by Lemma \ref{GradedNakayama} $j^*_y \mathcal{H}^{q_0}(\mathcal{E}) \neq 0$ and so $E_2^{0,q_0}$ survives in the spectral sequence. Hence $q_0=0$. We now consider $\mathcal{H}^0(\mathcal{E})$, and note that this must be flat over $Y$ to prevent $E_2^{-1,0}$ surviving in the spectral sequence.

If $\mathcal{H}^{q}(\mathcal{E} )\neq 0$ for some $q<0$ we can choose a maximal such value, $q_1$ then $\mathcal{H}^{q}(\E)$ is a non-zero $\mathbb{C}^*$-equivariant sheaf,  so $j^*_y \mathcal{H}^{q_1}(\mathcal{E})$ is nonzero by Lemma \ref{GradedNakayama}, and hence $E_2^{0,q_1}$ survives the spectral sequence. This cannot occur, hence $\mathcal{H}^i(\mathcal{E})=0$ for $i<0$.
\end{proof}

We set out the following notation for $\mathbb{C}^*$-equivariant base change by closed points. Let $\rho:X \rightarrow \Spec(D)$ be a flat $\mathbb{C}^*$-equivariant morphism of Noetherian schemes with $\mathbb{C}^*$-action which factors through a $\mathbb{C}^*$-equivariant projective morphism of Noetherian schemes $\pi:X \rightarrow \Spec(R)$ and flat affine morphism $\alpha:\Spec(R) \rightarrow \Spec(D)$.  Assume that $\Spec(D)$ has a good $\mathbb{C}^*$-action with unique $\mathbb{C}^*$-fixed closed point $d\in \Spec(D)$. Then there is a $\mathbb{C}^*$-equivariant closed immersion $i_d:\Spec(\mathbb{C}) \rightarrow \Spec(D)$ and we recall the diagram \eqref{ClosedBaseChange} noting that the morphisms are all $\mathbb{C}^*$-equivariant.

\begin{GradedBaseChangeCorollary}  \label{GradedBaseChangeCorollary} Suppose that $\E \in D^b(\Coh \, X)$ has $\mathbb{C}^*$-equivariant cohomology sheaves and $\mathbf{R}p_* \mathbf{L}j_d^* \E$ is a $\mathbb{C}^*$-equivariant sheaf on $\Spec(R_d)$. Then $\mathbf{R}\pi_{*} \E$ is a $\mathbb{C}^*$-equivariant coherent sheaf on $\Spec(R)$ flat over $\Spec(D)$.
\end{GradedBaseChangeCorollary}

\begin{proof}
As $\pi$ is projective it follows that $\mathbf{R}\pi_{*} \E \in D^b(R$-$\mod)$ by Lemma \ref{Bounded Push Down} and it has $\mathbb{C}^*$-equivariant cohomology sheaves as $p$ is $\mathbb{C}^*$-equivariant. It then follows from Lemma \ref{VitalLemma1} that if $\mathbf{L}l_d^*  \mathbf{R}\pi_{*} \E$ is a $\mathbb{C}^*$-equivariant coherent sheaf on $\Spec(R_d)$ for the fixed closed point $d \in \Spec(D)$ then $\mathbf{R}\pi_{*} \E$ is a $\mathbb{C}^*$-equivariant coherent sheaf on $\Spec(R)$ which is flat over $\Spec(D)$.

Hence we calculate $\mathbf{L}l_d^* \mathbf{R}\pi_{*} \E$. We note that as $a$ is affine $\mathbf{L}l_d^* \mathbf{R}\pi_{*} \E \in D^b(R$-$\mod)$ is a sheaf on $\Spec(R_d)$ if and only if $a_* \mathbf{L}l_d^* \mathbf{R}\pi_{*} \E$ is a sheaf on $\Spec(D/d)$. Then 
\begin{align*}
a_* \mathbf{L}l_d^* \mathbb{R}\pi_{*} \E & \cong \mathbf{L}i_d^* \alpha_{*} \mathbf{R}\pi_{*} \E \tag{Lemma \ref{Derived Base Change}}\\
&  \cong \mathbf{L}i_d^*  \mathbf{R}\rho_{*} \E  \\
&  \cong  \mathbf{R}q_{*} \mathbf{L}j_d^*   \E  \tag{Lemma \ref{Derived Base Change}} \\
& \cong  a_* \mathbf{R}p_{*} \mathbf{L}j_d^*   \E,
\end{align*}
which is a sheaf on $\Spec(D/d)$ since $\mathbf{R}\pi_{*} \mathbf{L}j_d^* \E$ is a sheaf on $\Spec ( R_d)$ by hypothesis.
\end{proof}

\section{$\mathbb{C}^*$-Equivariant deformations and tilting} \label{Graded Results}
We now expand the results of Section~\ref{Results} to include $\mathbb{C}^*$-equivariant deformations.

For the following proposition let $\pi:X \rightarrow \Spec(R)$ be a $\mathbb{C}^*$-equivariant projective morphism of finite type schemes over $\mathbb{C}$ with $\mathbb{C}^*$-actions such that the $\mathbb{C}^*$-action on $\Spec(R)$ is good and the scheme $X$ is normal. Define $\frak{R}$ to be the completion of $R$ with respect to the ideal $\mathfrak{m}$, and $\frak{R}$ is complete local with a good $\mathbb{C}^*$-action on $\Spec(\frak{R})$. Then define $\psi:\frak{X} \rightarrow \Spec \, \frak{R}$ to be the base change of $\pi$ by the morphism  $\iota: \Spec(\frak{R}) \rightarrow \Spec(R)$ induced by the completion map $R \rightarrow \frak{R}$, and there is a $\mathbb{C}^*$-action on $\frak{X}$ such that $\psi$ is $\mathbb{C}^*$-equivariant.

\begin{LiftingToEquiTilting} \label{GradedTiltingLift} A $\mathbb{C}^*$-equivariant tilting bundle $\frak{T}$ on $\frak{X}$ lifts uniquely to a $\mathbb{C}^*$-equivariant tilting bundle $T$ on $X$.
\end{LiftingToEquiTilting}

\begin{proof}
The morphism $\iota:\Spec(\frak{R}) \rightarrow \Spec(R)$ is induced by the completion of rings, hence is a flat affine morphism. Let  $\zeta:\frak{X} \rightarrow X$ denote the induced inclusion of fibre over $\frak{R}$. We let $i_{\mathfrak{m}}: \Spec(\mathbb{C}) \rightarrow \Spec(R)$ denote the inclusion of the unique $\mathbb{C}^*$-fixed closed point of $\Spec(R)$ and $j_{\mathfrak{m}}$ the inclusion of the corresponding fibre,  and we note $i_{\mathfrak{m}}$ and $j_{\mathfrak{m}}$ factor through $\iota$ and $\zeta$ respectively.

As the action on $\Spec(R)$ is good $\frak{T}$ lifts uniquely to a locally free $\mathbb{C}^*$-equivariant coherent sheaf $T$ on $X$ by \cite[Proposition A.6]{NamikawaFlopsandPoisson}. 

As $T$ is a locally free coherent sheaf it is a compact object of $D(X)$ and so satisfies Definition \ref{TiltingDef}(3). We now show that $T$ satisfies Definition \ref{TiltingDef}(1). In particular $\CalHom_{X}(T,T)$ is a $\mathbb{C}^*$-equivariant sheaf on $X$, hence $\mathbf{R}^k\Hom_X(T,T) =\Ext^k_X(T,T)$ is a $\mathbb{C}^*$-equivariant sheaf on $\Spec(R)$ and
\begin{align*}
\iota^* \Ext^k_X(T,T) & \cong \iota^*\mathbf{R}^k \pi_* \CalHom_X(T,T) \\
& \cong \mathbf{R}^k \psi_* \zeta^* \CalHom_{X}(T,T)  \tag{Flat base change} \\
& \cong \Ext^k_{\frak{X}}(\zeta^*T,\zeta^*T) \tag{Lemma \ref{DualPullback}} \\
& \cong \Ext^k_{\frak{X}}(\frak{T},\frak{T}). \tag{$\zeta^* T \cong \frak{T}$}
\end{align*}
Hence, as $\frak{T}$ is a tilting bundle on $\frak{X}$  it follows that  $\iota^*\Ext_X^k(T,T)=0$ when $k \neq 0$ so $i_{\mathfrak{m}}^* \Ext_X^k(T,T) =0$ and hence $\Ext_X^k(T,T) =0$ for $k \neq 0$ by Lemma \ref{GradedNakayama}.

Hence $T$ satisfies Definitions \ref{TiltingDef}(1) and (3) and so to prove that $T$ is a tilting bundle we need only check generation, Definition \ref{TiltingDef}(2), which we do now. We recall that $A=\End_{X}(T)$ and consider the adjoint functors  $T_*:=\mathbf{R}\Hom_{X}(T,-): D(X) \rightarrow D(A)$ and $T^*:= (-) \otimes_{A}^{\mathbf{L}} T: D(A) \rightarrow D(X)$. We will show that these are inverse equivalences by checking that the unit and counit are isomorphisms, and this will imply that $T$ generates $D(X)$ as under the equivalence $T$ corresponds to  $A$ which generates $D(A)$. The unit is automatically an isomorphism as $T_*T^*(A) \cong A$ and $A$ generates $D(A)$, so we are only required to show that the counit is an isomorphism. There is an $\mathbb{C}^*$-equivariant ample line bundle $\mathcal{L}$ on $X$ by Lemma \ref{Equivariant Ample Line Bundle}, and hence by Lemma \ref{Action Lemma} we have  $\mathbb{C}^*$-equivariant line bundles $\mathcal{L}^{\otimes m}$ for all $m \in \mathbb{Z}$. By Remark \ref{Equivariant Ample Line Bundle} $\{ \mathcal{L}^{\otimes m} \}$ generates $D(X)$ hence to show that the functors are an equivalence it is enough to check that the counit
\begin{equation*}
T^*T_* \mathcal{L}^{\otimes m} \xrightarrow{} \mathcal{L}^{\otimes m}
\end{equation*}
is a quasi-isomorphism for any $m$. In particular $T$ and $\mathcal{L}^{\otimes m}$ are $\mathbb{C}^*$-equivariant and hence the cone to this map, $C$, exists as a restriction from the $\mathbf{C}^*$-equivariant derived category, as noted in Remark \ref{ShiftedGradingRemark}.  We note that $T_*,T^*$ restrict to functors between $D^-( \Coh \, X)$ and $D^-(A$-$\mod)$ and hence $C$ can be realised as some bounded above complex of coherent sheaves whose cohomology sheaves can be given $\mathbb{C}^*$-equivariant structures
\begin{equation*}
T^*T_* \mathcal{L}^{\otimes m}  \xrightarrow{} \mathcal{L}^{\otimes m} \rightarrow C \rightarrow.
\end{equation*}
 Then, as the unit is an isomorphism $\mathbf{R}\Hom_{X}(T,C)=0$ and 
\begin{align*}
0 & \cong  \mathbf{L}\iota^* \mathbf{R}\Hom_{X}(T,C)\\ 
& \cong \mathbf{L} \iota^*\mathbf{R}\pi_{*} \RCalHom_{X}(T,C) \\
& \cong \mathbf{R} \psi_*  \mathbf{L} \zeta^*  \RCalHom_{X}(T,C)  \tag{Flat base change}\\
& \cong \mathbf{R}\psi_*   \RCalHom_{\frak{X}}(\mathbf{L}\zeta^* T, \mathbf{L}\zeta^* C) \tag{Lemma \ref{DualPullback}} \\
& \cong \mathbf{R}\Hom_{\frak{X}}(\frak{T},\mathbf{L}\zeta^*C). \tag{$\mathbb{L}\zeta^*\frak{T} \cong T$}
\end{align*}
 Hence $\mathbf{L}\zeta^* C=0$ as $\frak{T}$ is tilting bundle on $\frak{X}$. It then follows that $\mathbf{L}j_{\mathfrak{m}}^*C \cong 0$. Therefore $C$ is a $\mathbb{C}^*$-equivariant coherent sheaf on $X$ flat over $\Spec(R)$ by Lemma \ref{GradedVitalLemma1} and $j_{\mathfrak{m}}^*C \cong 0$ so $C \cong 0$ by Lemma \ref{GradedNakayama}. Hence we have proved that the functors are an equivalence and so $T$ is a tilting bundle on $X$.
\end{proof}

We now apply this result to $\mathbb{C}^*$-equivariant deformations. For the following two lemmas and theorem let $p:X_0 \rightarrow \Spec(R_0)$ be a $\mathbb{C}^*$-equivariant projective morphism of Noetherian schemes with $\mathbb{C}^*$-action, suppose that $\left(\rho:X \rightarrow \Spec \, D, j_d:X_0 \rightarrow X \right)$ is a $\mathbb{C}^*$-equivariant deformation of $X_0$, and assume that $\rho$ factors through a $\mathbb{C}^*$-equivariant projective morphism $\pi:X \rightarrow \Spec(R)$ followed by a $\mathbb{C}^*$-equivariant, flat, affine morphism $\alpha:\Spec(R) \rightarrow \Spec(D)$ such that $R\otimes_D D/d \cong R_0$.

We first check that $\mathbb{C}^*$-actions on tilting bundles lift to deformations over Noetherian, complete local rings.

\begin{GradedTiltingDeform} \label{GradedTiltingDeform}
Assume that $(D,d)$ is a Noetherian complete local ring. Then any $\mathbb{C}^*$-equivariant tilting bundle $T_0$ on $X_0$ lifts to a $\mathbb{C}^*$-equivariant tilting bundle $T$ on $X$. Further, $A:=\End_{X}(T)$ is a $D$-algebra and is flat as a $D$-module. 
\end{GradedTiltingDeform}
\begin{proof}
By Lemmas \ref{TiltingDeform} and \ref{BetterGeneration} we know that $T_0$ lifts to a tilting bundle $T$ on $X$ satisfying all of the requirements apart from possibly $\mathbb{C}^*$-equivariance. Hence we can assume we have a tilting bundle $T$ such that $A$ is flat as a $D$-module, but must still check that it has a $\mathbb{C}^*$-equivariant structure lifted from that on $T_0$. We do this now.

We let $\sigma_0,\beta_0$ denote the $\mathbb{C}^*$-action and projection maps for $X_0$ and $\sigma,\beta$ the $\mathbb{C}^*$-action and projection maps for $X$. As $T_0$ is $\mathbb{C}^*$-equivariant there is an isomorphism $\phi_0:\beta_0^*T \rightarrow \sigma_0^*T$ and we aim to produce an isomorphism $\phi: \beta^*T \rightarrow \sigma^*T$ realising a $\mathbb{C}^*$-equivariant structure on $T$. Firstly, for the locally free coherent sheaf $\sigma_0^*T \cong \beta_0^*T$ we note that $\Ext^i_{X_0 \times \mathbb{C}^*}(\sigma_0^*T_0,\sigma_0^*T_0) \cong  \Ext^i_{X_0 \times \mathbb{C}^*}(\beta_0^*T_0,\beta_0^*T_0) \cong 0$ for $i \neq 0$ as $\beta_0, \sigma_0$ are flat and $T_0$ is a tilting bundle. Hence by Lemma \ref{TiltingDeform} both $\sigma_0^*T_0$ and $\beta_0^*T_0$ lift uniquely to tilting bundles on $X \times \mathbb{C}^*$, which we denote by $T_{\sigma}$ and $T_{\beta}$, and hence there must be a canonical isomorphism between them $\theta: T_{\beta} \rightarrow T_{\sigma}$ as the lifts are unique. Further, both $\sigma^*T$ and $\beta^*T$ are lifts of $\sigma_0^*T_0$ and $\beta_0^*T_0$, and as these lifted uniquely there are canonical isomorphisms $\theta_\sigma: T_{\sigma} \rightarrow  \sigma^*T $ and $\theta_{\beta}: \beta^*T \rightarrow T_{\beta}$ and hence we have constructed an isomorphism
\[
\phi: \beta^*T \xrightarrow{\theta_{\beta}} T_{\beta} \xrightarrow{\theta} T_{\sigma} \xrightarrow{\theta_{\sigma}} \sigma^*T.
\] To check that this defines a $\mathbb{C}^*$-equivariant structure on the vector bundle $T$ we must also check the two conditions of Definition \ref{C* Action}. To check it satisfies the condition $(1 \times id_X)^* \phi= id_{T}$ we consider $\chi:= (1 \times id_{X})^* \phi -  id_{T} \in \Hom_{X}(T,T)$ and note that  $j_d^*(\chi)=0$ as $\phi_0$ induced a $\mathbb{C}^*$-equivariant structure on $T_0$. We then consider the cokernel of $\chi$ 
\begin{equation*}
T \xrightarrow{\chi} T \rightarrow C \rightarrow 0.
\end{equation*}
Applying $j_d^*$ we find
\begin{equation*}
T_0 \xrightarrow{0} T_0 \rightarrow j_d^*C \rightarrow 0.
\end{equation*} 
and deduce that $T_0 \cong j_d^*C$, and hence as the lift was unique $T \cong C$ so  $\chi$ is zero and $(1 \times id_{X})^* \phi = id_{T}$. A similar argument shows that $\beta_{23}^*\phi \circ  (id_{\mathbb{C}^*} \times \sigma)^*\phi =(m \times id_{X})^*\phi$, and hence $\phi$ defines a $\mathbb{C}^*$-equivariant structure on $T$.
\end{proof}

We now expand our results to include deformations which are not just infinitesimal but have a base space of finite type with good $\mathbb{C}^*$-action. For the following lemma and theorem we add the extra assumptions that the $\mathbb{C}^*$-actions on $\Spec(D)$ and $\Spec(R)$ are good, that $(D,d)$ is a graded complete local ring of finite type over $\mathbb{C}$, and that $X$ is normal. 

\begin{GradedGeneration} \label{GradedGeneration} 
Any $\mathbb{C}^*$-equivariant tilting bundle $T_0$ on $X_0$ lifts to a $\mathbb{C}^*$-equivariant tilting bundle $T$ on $X$.  Further, $A:=\End_{X}(T)$ is a graded $D$-algebra which is flat as  $D$-module. 
\end{GradedGeneration}
\begin{proof}

We define $\frak{D}$ to be the completion of $D$ with respect to its unique maximal ideal $d$, and define $\frak{R}:= R \otimes_{D} \frak{D}$ and $\frak{X}:=X \times_{\Spec(D)} \frak{D}$. It follows that $T_0$ lifts to a $\mathbb{C}^*$-equivariant tilting bundle on $\frak{X}$ by Lemma \ref{GradedTiltingDeform}, and then it follows that $T$ lifts to a $\mathbb{C}^*$-equivariant tilting bundle $T$ on $X$ by Proposition \ref{GradedTiltingLift}.  We now check that $\Hom_{X}(T,T)$ is flat over $\Spec(D)$. To do this we make use of the unique $\mathbb{C}^*$-invariant closed point $d \in \Spec(D)$ and then the argument is similar to the proof of Lemma \ref{TiltingDeform}. We consider the diagram (\ref{ClosedBaseChange}) for the inclusion of $d$ and then $\mathbf{L} j_d^* \RCalHom_{X}(T,T) \cong \RCalHom_{X_0}(\mathbf{L} j_d^*T,\mathbf{L} j_d^*T)$ by Lemma \ref{DualPullback}. Hence as $\mathbf{R}p_*\RCalHom_{X_0}(\mathbf{L}j_d^*T,\mathbf{L}j_d^*T) \cong \mathbf{R}\Hom_{X_0}(T_0,T_0)$ is a $\mathbb{C}^*$-equivariant coherent sheaf on $\Spec(R_0)$ it follows by Corollary \ref{GradedBaseChangeCorollary} that $\mathbf{R}\Hom_{X}(T,T)$ is a $\mathbb{C}^*$-equivariant coherent sheaf on $\Spec(R)$ and is flat over $\Spec(D)$. In particular both $\Ext^i_{X}(T,T)=0$ for $i \neq 0$ and $A=\End_{X}(T)$ is flat as a $D$-module. It is graded by Lemma \ref{Action Lemma}.
\end{proof}

Hence the graded version of Theorem \ref{Main} holds. 

\begin{Graded Main} \label{Graded Main}
Any $\mathbb{C}^*$-equivariant tilting bundle $T_0$ on $X_0$ lifts to a $\mathbb{C}^*$-equivariant tilting bundle $T$ on $X$. Further, $T_0$ defines the graded ring $A_0=\End_{X_0}(T_0)$ and $T$ defines the graded ring $A:=\End_{X}(T)$ which is a graded $D$-algebra and flat as a $D$-module, there is a graded map $A \rightarrow A_0$ such that $A \otimes_{D} D/d \cong A_0$, and this map and isomorphism define $D\rightarrow A$ as a graded deformation of $A_0$. 
\end{Graded Main}

\begin{proof}
By Lemma \ref{GradedGeneration} a $\mathbb{C}^*$-equivariant tilting bundle $T_0$ on $X_0$ lifts to a $\mathbb{C}^*$-equivariant tilting bundle $T$ on $X$. In particular, the $\mathbb{C}^*$-equivariant flat morphism $\rho:X \rightarrow \Spec(D)$ produces a graded map $\theta:D \rightarrow A$ defining $A$ as a flat and graded $D$-algebra. The $\mathbb{C}^*$-equivariant map $j_d:X_0 \rightarrow X$ defines a unit map $T \rightarrow j_{d*}j_d^*T$ and hence as $j_d^*T\cong T_0$ applying $\Hom_X(T,-)$ produces a graded morphism \[u_d:A= \Hom_{X}(T,T) \rightarrow \Hom_{X}(T,j_{d*}j_d^* T) \cong \Hom_{X_0}(j_d^*T,j_d^*T)\cong A_0.\] Then we recall the diagram \eqref{ClosedBaseChange} corresponding to the closed point $d \in \Spec(D)$, and
\begin{align*}
A \otimes_{D}D/d & := \mathbf{L}i_d^*\mathbf{R}\rho_{*} \RCalHom_{X}(T,T) \tag{$A$ is flat over $D$} \\
&\cong \mathbf{R}q_{*} \mathbf{L}j_d^* \RCalHom_{X}(T,T) \tag{Lemma \ref{Derived Base Change}}\\
&  \cong  \mathbf{R}\Hom_{X_0}(\mathbf{L}j_d^* T, \mathbf{L}j_d^*T)  \tag{Lemma \ref{DualPullback}}\\
 &\cong \mathbf{R}\Hom_{X_0}(T_0,T_0)=A_0 \tag{$j_d^*T \cong T_0$}.
\end{align*}
This isomorphism of graded rings $A \otimes_D D/d \cong A_0$ is uniquely defined by the $\mathbb{C}^*$-equivariant isomorphisms $j_d \times_{\Spec(D)} \Spec(D/d):X_0 \rightarrow X \times_{\Spec(D)} \Spec(D/d) $ and $j_d^*T \cong T_0$. Hence this defines $(\theta:D \rightarrow A, u_d:A \rightarrow A_0)$, a graded deformation of $A_0$.
\end{proof}

We can now consider closed points in $\Spec(D)$ other than the unique $\mathbb{C}^*$-invariant closed point.

\begin{FilteredCorollary} \label{FilteredCorollary} Suppose that we are in the situation of Theorem \ref{Graded Main} such that $A_0$ is positively graded ring, and suppose that $j_z:z \hookrightarrow \Spec(D)$ is the inclusion of a closed point. Then $T_z:=j_z^*T$ is a tilting bundle on the fibre $X_z:=X \times_{D} z$ such that $A_z:=\End_{X_z}(T_z) \cong A \otimes_D D/z$ is a filtered algebra with associated graded algebra isomorphic to $A_0$.
\end{FilteredCorollary}
\begin{proof}
It follows from \cite[Proposition 2.9]{BHTilting} that $T_z$ is a tilting bundle on $X_z$ and $A_z:=\Hom_{X_z}(T_z,T_z) \cong A \otimes_D D/z$. By Lemma \ref{ReesRingGradedDeformations} and Remark \ref{ReesFibresRemark} $A_z$ is a filtered algebra whose associated graded algebra is isomorphic to $A_0$.
\end{proof}

\section{Application: deformations of resolutions of singularities} \label{Applications}

In this section we consider minimal resolutions of rational surface singularities and crepant resolutions of symplectic quotient singularities, and in both cases we will recall the existence of certain classes of deformations and tilting bundles in order to apply Theorem \ref{Main} or \ref{Graded Main}. In the case of rational surface singularities this will produce a new class of deformations of the reconstruction algebras of \cite{WemyssGL2} which will provide a noncommutative simultaneous resolution of the Artin component. For surface quotient singularities these algebras are graded. In the case of symplectic quotient singularities this will produce graded deformations of the skew group algebras whose fibres recover symplectic reflection algebras for certain parameters.

\subsection{Resolutions of singularities}
 We begin by recalling definitions relating to resolution of singularities.
\begin{resolution}
 Let $Y$ be a (possibly singular) variety. A smooth variety $X$ with a projective birational map $\pi: X \rightarrow Y$ that is bijective over the smooth locus of $Y$ is called a \emph{resolution} of $Y$. A resolution $X$ is a \emph{minimal resolution} of $Y$ if any other resolution factors through it. In general minimal resolutions do not exist, but they always exist for surfaces \cite[Corollary 27.3]{LipmanRationalSingularities}. If we assume that $X$ and $Y$ are normal then a resolution $X$ is a \emph{crepant resolution} of $Y$ if $\pi^* \omega_Y =\omega_{X}$, where $\omega_{X}$ and $\omega_Y$ are the canonical classes of $X$ and $Y$. In general crepant resolutions do not exist. A singularity $Y$ is \emph{rational} if for any resolution $\pi:X\rightarrow Y$
\begin{equation*}
\mathbf{R} \pi_{*} \mathcal{O}_{X} \cong \mathcal{O}_{Y}.
\end{equation*}
If this holds for one resolution it holds for all resolutions \cite[Lemma 1]{Viehweg}.

Let $V=\mathbb{C}^n$ and $G<\GL(V)$ be a finite group. The group $G<\GL(V)$ is $\emph{small}$ if it contains no non-trivial pseudo-reflections. There is a corresponding \emph{quotient singularity} $V/G$ which is affine with ring of functions
\begin{equation*}
R:=\mathbb{C}[V]^G.
\end{equation*}
This has a natural grading given by identifying $\mathbb{C}[V]=\mathbb{C}[x_1, \dots x_n]$ with each $x_i$ in grade 1, and this defines a good $\mathbb{C}^*$-action on $\Spec(R)$. Quotient singularities are also known to be rational singularities, \cite[Proposition 1]{Viehweg}.
\end{resolution}

\subsection{Rational surface singularities} We discuss minimal resolutions of rational surface singularities. We recall the existence of deformations, then of tilting bundles, and then combine these in order to apply Theorem \ref{Main} in this ungraded case. We then restrict to studying surface quotient singularities, which are examples of rational surface singularities which carry a $\mathbb{C}^*$-action as described above, and in this situation can apply Theorem \ref{Graded Main} and produce graded deformations of the reconstruction algebras. 

\subsubsection*{Rational surface singularities} \label{Deformations of Rational Surface Singularities}
We first consider when $R_0$ is a complete local, Noetherian $\mathbb{C}$-algebra such that $\Spec(R_0)$ is a rational surface singularity and $p: X_0 \rightarrow \Spec(R_0)$ its minimal resolution. We recall the existence of deformations and tilting bundles for minimal resolutions of rational surface singularities in order to apply Theorem \ref{Main}. 

A rational surface singularity is normal and hence isolated, and so $\Spec(R_0)$ has a versal deformation \cite{ExistenceofVersal}. There is a unique component of the versal deformation for which simultaneous resolutions exist after base change, the \emph{Artin component} introduced in \cite{Artin}, which is smooth \cite[Theorem 1]{WahlSimultaneous}. The simultaneous resolution of the Artin component (after base change) can be realised as the versal deformation of the minimal resolution $p :X_0 \rightarrow \Spec(R_0)$ \cite[Corollary 4.6]{Artin},  which we denote $\left(\rho:X \rightarrow \Spec \, D, j_d:X_0 \rightarrow X \right)$. This deformation $\rho$ factors through the simultaneous resolution $\pi:X \rightarrow \Spec(R)$ where $R$ is the base change of the Artin component. Indeed, there is an action of a Weyl group on $\Spec(D)$ corresponding to the Dynkin diagram configurations of the -2 curves in the exceptional divisor of $\pi_0 :X_0 \rightarrow \Spec(R_0)$, and the base change map from $\Spec(R)$ to the Artin component is induced by the quotient map for this Weyl group action \cite[Theorem 1]{WahlSimultaneous}.

We now also recall the existence of tilting bundles on the minimal resolution $X_0$. 

\begin{Reconstruction}[{\cite[Theorem 1.1, Lemma 3.6]{WemyssGL2}}] \label{Reconstruction}
There exists a tilting bundle $T_0$ on $X_0$ defined as 
\begin{equation*}
T_0 := \bigoplus_{M_i}  \frac{p^* M_i}{\text{torsion}}
\end{equation*}
where the sum is taken over the indecomposable special maximal Cohen-Macaulay modules of $R_0$.
\end{Reconstruction}
We recall that a maximal Cohen-Macaulay $R_0$-module $M$ is \emph{special} if $\Ext^1_{R_0}(M,R_0)=0$  \cite[Theorem 2.7]{ClassificationOfCM}.

\begin{RCA} \label{RCA}
 The \emph{reconstruction algebra} is defined as
\begin{equation*}
A_0:=\End_{X_0}(T_0).
\end{equation*}
\end{RCA}
We consider $A_0$ as a quiver with relations, which can be calculated from the representation theory of $R_0$ as $A_0$ is isomorphic to $\End_{R_0}(\oplus_i M_i)$ \cite[Lemma 3.6]{WemyssGL2}, and $p_*T_0 = \bigoplus_i M_i$ \cite[Lemma 2.2]{Esnault}.

By the above discussion we know there exists a versal deformation of $X_0$ and a tilting bundle on $X_0$ which produces the reconstruction algebra. We can now apply Theorem \ref{Main} to find deformations of the reconstruction algebra derived equivalent to the versal deformation.

\begin{RationalMain} \label{RationalMain} Let $\Spec(R_0)$ be a complete local rational surface singularity, $p:X_0 \rightarrow \Spec(R_0)$ be its minimal resolution, and $A_0=\End_{X_0}(T_0)$ its corresponding reconstruction algebra. Consider the versal deformation $(\rho:X \rightarrow \Spec\, D, j_d:X_0 \rightarrow X)$ of $X_0$. Then;
\begin{enumerate}
\item The tilting bundle $T_0$ on $X_0$ lifts to a tilting bundle $T$ on $X$. 
\item The algebra $A:=\End_{X}(T)$ is derived equivalent to the versal deformation space $X$.
\item The algebra $A$ is a deformation of $A_0$ over $D$ defined by the deformation $(X \rightarrow \Spec \, D, j_d:X_0 \rightarrow X)$. 
\end{enumerate}
\end{RationalMain}

\subsubsection*{Surface quotient singularities}
We now recall the existence of $\mathbb{C}^*$-equivariant deformations and tilting bundles for minimal resolutions of surface quotient singularities in order to apply Theorem \ref{Graded Main}. Here $G$ will be a finite, small subgroup of $\GL_2(\mathbb{C})$, we will let $R_0=\mathbb{C}[x,y]^G$ be the graded invariant ring with $x,y$ in degree one such that $\Spec(\mathbb{C}[x,y]^G)=\mathbb{C}^2/G$ is the corresponding quotient surface singularity with good $\mathbb{C}^*$-action, and we also  note that $X_0$ can be realised as $G$-$\Hilb(\mathbb{C}^2)$ by \cite{McKaycorrespondenceIshii} and this defines a $\mathbb{C}^*$-action on $X_0$ such that $p$ is $\mathbb{C}^*$-equivariant. In particular, the completion of this situation with respect to the unique $\mathbb{C}^*$-equivariant point of $\mathbb{C}[x,y]^G$ produces the complete local setting, and the methods in the appendix of \cite{NamikawaFlopsandPoisson} can recover this finite type setting from the complete local one using the good $\mathbb{C}^*$-action. 

Suppose that $\mathbb{C}^2/G$ is a surface quotient singularity. Then the versal deformations of $\mathbb{C}^2/G$ and $X_0$ are $\mathbb{C}^*$-equivariant by \cite{RimStructure}. In particular, the $\mathbb{C}^*$-action on $X_0$  lifts to a unique (up to equivariant isomorphism) $\mathbb{C}^*$-action on the versal deformation of $X_0$. We use the notation $(\rho: X \rightarrow \Spec(D),j_d:X_0 \rightarrow X)$ for this versal deformation. By the results of \cite{Brieskorn} and \cite{Laufer} a quotient singularity $\mathbb{C}^2/G$ is taut and this implies that the action on $\Spec(D)$ and $\Spec(R)$ is good, see \cite[Section 3.7]{Pinkham}. 

The following tangent space calculation allows us to deduce the dimension of the versal deformation of $X_0$ as the tangent space to $\Spec(D)$ at $d$ is $\CH^1(X_0,\mathcal{T}_0)$ where $\mathcal{T}_{X_0}$ is the tangent bundle of $X_0$.

\begin{HHSQS} \label{HHSQS} Let $X_0 \rightarrow \Spec(\mathbb{C}[x,y]^G)$ be the resolution of a surface quotient singularity with exceptional divisor $E=\sum E_i$ where $E_i \cong \mathbb{P}^1$ and $E_i^2 = -a_i$. Then
\begin{equation*} \dim \CH^j(X_0,\mathcal{T}_{X_0})= \left\{ \begin{array}{c c} \sum_i (a_i-1) & \text{ if } j=1; \\
 0 & \text{ if } j>1. \end{array} \right.
\end{equation*}
Moreover, the $R_0$-module $\CH^1(X,\mathcal{T}_{X_0})$ is graded in strictly negative degree.
\end{HHSQS}
\begin{proof}
We note that as the resolution has fibres of dimension $\le 1$  $\dim \CH^j(X_0,\mathcal{T}_{X_0})=0$ for $j>1$ by \cite[Proposition 5.2.34]{AG and arithmetic curves}. Then $\dim \CH^1(X_0,\mathcal{T}_{X_0})= \sum_i (a_i-1)$ by \cite[Equation 5]{BehnkeRiemenschneider} as it is known that all surface quotient singularities are taut,  \cite{Brieskorn} and \cite{Laufer}. The fact that $\CH^1(X_0,\mathcal{T}_{X_0})$ is graded in strictly negative degree also follows from tautness \cite[Section 3.7]{Pinkham}.
\end{proof}

We now note the existence of $\mathbb{C}^*$-equivariant tilting bundles on the minimal resolution. Theorem \ref{Reconstruction} produces a tilting bundle on the completion of $X$, and this tilting bundle can be given a $\mathbb{C}^*$-equivariant structure by considering the maximal Cohen-Macaulay modules as graded submodules of $\mathbb{C}[x,y]$, hence by Proposition \ref{GradedTiltingLift} this lifts to a $\mathbb{C}^*$-equivariant tilting bundle on $X$.

\begin{GradedReconstruction} \label{GradedReconstruction}
There exists a $\mathbb{C}^*$-equivariant tilting bundle $T_0$ on $X_0$ defined as 
\begin{equation*}
T_0 := \bigoplus_{M_i}  \frac{\pi^* M_i}{\text{torsion}}
\end{equation*}
where the sum is taken over the indecomposable graded special maximal Cohen-Macaulay modules of $\mathbb{C}[x,y]^G$ that are graded submodules of $\mathbb{C}[x,y]$. 
\end{GradedReconstruction}

\begin{GradedReconstructionDefinition}
We define the \emph{(McKay) graded} reconstruction algebra to be 
\begin{equation*}
A_0:=\Hom_{X_0}(T_0,T_0).
\end{equation*}

\end{GradedReconstructionDefinition}
As a graded algebra $A_0 \cong \End_{\mathbb{C}[x,y]^G} \left( \bigoplus M_i \right)$, where the sum is taken over the indecomposable special maximal Cohen-Macaulay modules that are graded submodules of $\mathbb{C}[x,y]$.

The vertices of the McKay quiver correspond to indecomposable maximal Cohen-Macaulay modules and the graded reconstruction algebra is the full subquiver consisting of vertices corresponding to special maximal Cohen-Macaulay modules (after some choice has been made to identify representations with MCM $\mathbb{C}[x,y]^G$-modules). As such $A_0$ has a positive grading induced from path length in the McKay quiver. This description of the reconstruction algebras is given in \cite[Section 4]{WemyssGL2}.

\begin{NonUniqueGradingRemark}
This choice of grading is not unique, but is motivated by the standard grading on $\mathbb{C}[x,y]$ and so the $\mathbb{C}^*$-action on $\mathbb{C}^2/G$. Consider the simplest case of a type $A_1$ surface singularity. Then $\mathbb{C}[x,y]^{\frac{1}{2}(1,1)}=\mathbb{C}[A,B,C]/(AB-C^2)$. The two indecomposable maximal Cohen Macaulay modules are $(1)$ and $(A,C)$. We note that our choice of grading is such that the module $(1)$ is generated in degree 0 and $(A,C)\cong (x,y)$ is generated in degree 1. This is not the same grading as if we were to consider $(A,C)$ as a graded ideal of $\mathbb{C}[A,B,C]/(AB-C^2)$.
\end{NonUniqueGradingRemark}

By the discussions above we know there exist $\mathbb{C}^*$-deformations of $X_0$ and $\mathbb{C}^*$-equivariant tilting bundles on $X_0$ which produce the graded reconstruction algebras. We can now apply Theorem \ref{Graded Main} to find graded deformations of the graded reconstruction algebras.

\begin{GL2Main} \label{GL2Main} Let $G$ be a small, finite subgroup of $\GL_2(\mathbb{C})$, let $\mathbb{C}^2/G$ be the corresponding quotient singularity, let $p:X_0 \rightarrow \mathbb{C}^2/G$ be the minimal resolution, and let $A_0=\End_{X_0}(T_0)$ the corresponding graded reconstruction algebra, all with $\mathbb{C}^*$-actions and gradings as discussed above. Consider the versal deformation $(\rho: X \rightarrow \Spec \, D, j_d:X_0 \rightarrow X)$ of $X_0$ with its $\mathbb{C}^*$-equivariant structure. Then:
\begin{enumerate}
\item The $\mathbb{C}^*$-equivariant tilting bundle $T_0$ on $X_0$ lifts to a $\mathbb{C}^*$-equivariant tilting bundle $T$ on $X$.
\item The algebra $A:=\End_X(T)$ is derived equivalent to the versal deformation total space $X$.
\item The algebra $A$ is a graded deformation of $A_0$ over $D$ induced by the $\mathbb{C}^*$-equivariant versal deformation $(\rho: X \rightarrow \Spec \,D, j_d:X_0 \rightarrow X)$.
\item For any closed point $z \in \Spec(D)$ the fibres $A_z:=A \otimes_D (D/z)$ are filtered algebras with the property that the associated graded algebra of $A_z$ is isomorphism to $A_0$.
\end{enumerate}
\end{GL2Main}
Hence we have constructed graded deformations of the reconstruction algebra. In the special case when $G \subset \SL_2(\mathbb{C})$ the minimal resolutions are crepant, the reconstruction algebra is the preprojective algebra, and the fibres found are the deformed preprojective algebras of \cite{CBH} which are module-finite over their centres.
\subsection{Example: deformations of reconstruction algebras of type $\mathbb{A}$}
In this section we briefly recall the definition of the reconstruction algebras of type $\mathbb{A}$. We then apply the above results to calculate the deformations of the reconstruction algebras corresponding to the Artin component, and calculate some explicit examples. 

Reconstruction algebras of type $\mathbb{A}$ are those arising from tilting bundles on the minimal resolutions of quotient singularities for the finite groups
\begin{equation*}
\frac{1}{r}(1,a):= \left\langle \left( \begin{array}{c c} \varepsilon_r & 0 \\ 0 & \varepsilon^a_r \end{array} \right) \right \rangle < \GL_2(\mathbb{C})
\end{equation*} 
 with $\varepsilon_r$ a primitive $r^{th}$ root of unity and $a$ coprime to $r$. In this case the sequence of self intersection numbers of the exceptional divisor can be calculated from the Hirzebruch-Jung continued fraction $r/a=[a_1,a_2,\dots, a_k]$ and the sequence $r/(r-a)=[b_1,b_2, \dots,b_l]$ can be used to identify $\mathbb{C}[x,y]^G$ with the polynomial ring $\mathbb{C}[Z_0, \dots Z_{l+1}]$ modulo the relations
\begin{equation*}
Z_i Z_{j+1} = Z_{i+1} \left( \prod_{\beta=i+1}^{j}  Z_{\beta}^{b_{\beta}-2} \right)  Z_j \quad \text{for}\quad 0 \le i <j \le l.
\end{equation*}
This set of relations can be visualised as the $2$ by $2$ quasi-determinants of the matrix
\begin{equation*}
\left( \begin{array}{c c c c c c c} Z_0 & & Z_1  & & \dots & & Z_{l} \\
 & Z_1^{b_1-2} & &  Z_2^{b_2-2} & & Z_{l}^{b_{l}-2} & \\
Z_1 & & Z_2 & & \dots & & Z_{l+1}
\end{array} 
\right)
\end{equation*}
where the 2 by 2 quasi-determinant for columns $i$ and $j$ is exactly the relation outlined above, see \cite[Satz 2,3]{RiemenschneiderNach}.

The special indecomposable maximal Cohen-Macaulay modules can be described as the trivial module $(1)=\mathbb{C}[x,y]^G$ and $\left(Z_i, Z_{i+1}^k \right)$ for $0\le i \le l$ and $1 \le k \le b_{i+1}-1$, where we note the isomorphism between the modules
\begin{equation*}
\left( Z_i,Z_{i+1}^{b_{i+1}-1} \right) \xrightarrow{\phi_i} \Big( Z_{i+1},Z_{i+2} \Big) 
\end{equation*}
defined by multiplication by $Z_{i+1}/Z_i=Z_{i+2}/Z_{i+1}^{b_i-1}$, see \cite[Appendix A1]{WunramReflexiveModules}. We note that the grading on these modules is determined by identifying them as submodules of $\mathbb{C}[x,y]$. The reconstruction algebra is then the endomorphism ring of the sum of these modules, 
\begin{equation*} A_0=\End_{\mathbb{C}[x.y]^G} \left( \left(1\right) \oplus\bigoplus_{i=0}^{l} \bigoplus_{k=1}^{b_{i+1}-2} \left(Z_i, Z_{i+1}^k \right) \right).
\end{equation*} We refer the reader to \cite{RCAA} for a more detailed description of these algebra, including an explicit presentation as the path algebra of a quiver with relations.

\begin{Reconstruction Algebra Example 12,7}[$G=\frac{1}{12}(1,7)$]
In this case the Hirzebruch-Jung continued fractions are $12/7=[2,4,2]$ and $12/5=[3,2,3]$, and the ring $\mathbb{C}[x,y]^G$ is generated by $Z_0,Z_1,Z_2,Z_3,Z_4$ subject to the relations given by the 2 by 2 quasi-determinants of
\begin{equation*}
\left( \begin{array}{c c c c c c c} Z_0 & & Z_1  & & Z_2 & & Z_3 \\
 & Z_1 & &   & & Z_{3} & \\
Z_1 & & Z_2 & & Z_3 & & Z_{4}
\end{array} 
\right).
\end{equation*}
There are 4 indecomposable special maximal Cohen Macaulay modules, $M_0:=(1)$, $M_1:=(Z_0,Z_1) \cong (x^7,y)$, $M_2:=(Z_0,Z_1^2)\cong (Z_1,Z_2) \cong (Z_2,Z_3) \cong (x^2,y^2)$, $M_3:=(Z_2,Z_3^2) \cong(Z_3,Z_4) \cong (x,y^7)$. This algebra can be presented as the following quiver with relations
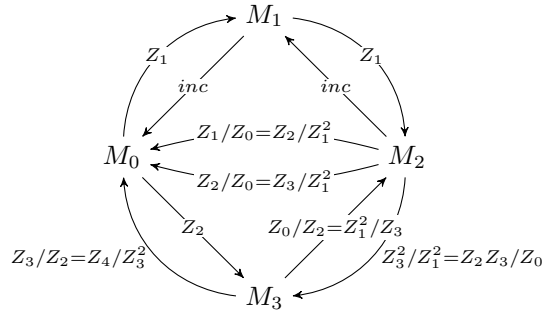
\begin{align*}
\begin{tikzpicture}[scale =0.75]
\node (C0) at (0,0)  {$M_0$};
\node (C2) at (5,0)  {$M_2$};
\node (C1) at (2.5,2.5)  {$M_1$};
\node (C3) at (2.5,-2.5)  {$M_3$};
\draw [->,bend left=40] (C0) to node[gap]  {$\scriptstyle{Z_1}$} (C1);
\draw [->,bend left=40] (C1) to node[gap]  {$\scriptstyle{Z_1}$} (C2);
\draw [->,bend left=40] (C2) to node[right]  {$\scriptstyle{Z_3^2/Z_1^2=Z_2Z_3/Z_0}$} (C3);
\draw [->,bend left=40] (C3) to node[left]  {$\scriptstyle{Z_3/Z_2=Z_4/Z_3^2}$} (C0);
\draw [->] (C1) to node[gap]  {$\scriptstyle{inc}$} (C0);
\draw [->] (C2) to node[gap]  {$\scriptstyle{inc}$} (C1);
\draw [->] (C3) to node[gap]  {$\scriptstyle{Z_0/Z_2=Z_1^2/Z_3}$} (C2);
\draw [->] (C0) to node[gap]  {$\scriptstyle{Z_2}$} (C3);
\draw [->,bend right=15] (C2) to node[gap]  {$\scriptstyle{Z_1/Z_0=Z_2/Z_1^2}$} (C0);
\draw [->,bend left=15] (C2) to node[gap]  {$\scriptstyle{Z_2/Z_0=Z_3/Z_1^2}$} (C0);
\end{tikzpicture}
\end{align*}
where the relations are those implied by the labelling.
\end{Reconstruction Algebra Example 12,7}

We now outline how to calculate the deformation of reconstruction algebras of type $\mathbb{A}$ corresponding to the simultaneous resolution of the Artin component. The Artin component of the versal deformation space (after base change) of a cyclic quotient singularity can be realised as the polynomial algebra generated by $Z_i^{(j)}$ for $0\le i \le l+1$ and $1 \le j \le b_i$ subject the relations given as  2 by 2 quasi-determinants of the matrix
\begin{equation*}
\left( \begin{array}{c c c c c c c} Z^{(1)}_0 & & Z^{(b_1)}_1  & & \dots & & Z^{(b_l)}_{l} \\
 & \prod_{1<j<b_1} Z_1^{(j)} & &  \prod_{1<j<b_2} Z_2^{(j)}  & & \prod_{1<j<b_l} Z_1^{(j)}  & \\
Z_1^{(1)} & & Z_2^{(1)} & & \dots & & Z^{(1)}_{l+1}
\end{array} 
\right).
\end{equation*}
We denote this algebra by $R$, and it is described in \cite[Satz 7 and Section 7]{RiemenschneiderNach}. This algebra carries an action of the Weyl group $S_{b_1-1} \times S_{b_2-2} \times S_{b_3-2} \times \dots \times S_{b_l -1}$ where the $i^{th}$ symmetric group in the product acts by permuting the $b_{i}-2$ generators $Z_i^{(j)}$ for $1<j<b_i$ listed in the central row, with the first and last symmetric group also acting  on $Z_1^{(1)}$ and $Z_l^{(b_l)}$ respectively.

\begin{ExplicitCalculationRemark} We note that to explicitly calculate $A$ it is sufficient to work on $R$ as by flat base change $\pi_* T$ must correspond to some lift $\bigoplus M_i'$ of $\bigoplus M_i$ from $R_0$ to $R$, and then $A \cong \End_R(\bigoplus M_i')$ as in the case of the reconstruction algebra.
\end{ExplicitCalculationRemark}

The $\mathbb{C}[x,y]^G$-module $\left(Z_i, Z_{i+1}^k \right)$ lifts to an $R$-module $\left( Z_i^{(b_i)}, Z_{i+1}^{(1)} \dots Z_{i+1}^{(k)} \right)$  for $0\le i \le l$ and $1 \le k \le b_{i+1}-1$, and there are still isomorphisms
\begin{equation*}
\left( Z_i^{(b_i)}, Z_{i+1}^{(1)} \dots Z_{i+1}^{(b_{i+1}-1)} \right) \xrightarrow{\phi_i} \left( Z^{(b_i+1)}_{i+1},Z^{(1)}_{i+2} \right) 
\end{equation*}
defined by multiplication by $Z^{(b_{i+1})}_{i+1}/Z_i^{(b_i)}=Z^{(1)}_{i+2}/Z_{i+1}^{(1)} \dots Z_{i+1}^{(b_{i+1}-1)}$. These modules are graded in the same degree as the modules they lift. Then the deformed graded reconstruction algebra can be described as
\begin{equation*} A=\End_{R} \left( \left(1\right) \oplus\bigoplus_{i=0}^{l} \bigoplus_{k=1}^{b_{i+1}-2}\left( Z_i^{(b_i)}, Z_{i+1}^{(1)} \dots Z_{i+1}^{(k)} \right) \right).
\end{equation*} 

\begin{RCAExample}($G=\frac{1}{12}(1,7)$)
We continue with the previous example. The graded deformation of the singularity after base change can be viewed as $\mathbb{C}[Z_i^{(j)}]$ subject to relations defined by the 2 by 2 quasi-determinants of 
\begin{equation*}
\left( \begin{array}{c c c c c c c} Z^{(1)}_0 & & Z^{(3)}_1  & & Z^{(2)}_2 & & Z^{(3)}_3 \\
 & Z^{(2)}_1 & & & & Z_{3}^{(2)} & \\
Z^{(1)}_1 & & Z^{(1)}_2 & & Z^{(1)}_3 & & Z^{(1)}_{4}
\end{array} 
\right).
\end{equation*}

In this case the modules corresponding to $\pi_*T$ are $M'_0=(1)$, $M'_1=(Z_0^{(1)},Z_1^{(1)})$, $M'_2=(Z_0^{(1)},Z_1^{(1)}Z_1^{(2)}) \cong (Z_1^{(3)},Z_2^{(1)})\cong (Z_2^{(2)},Z_3^{(1)})$, and $M'_3 := (Z_2^{(2)},Z_3^{(1)}Z_3^{(2)}) \cong (Z_3^{(3)},Z_4^{(1)})$, which are generated in the same degrees as the modules they restrict to in the undeformed case. Then in this case the deformed reconstruction algebra can be realised as the path algebra of the quiver, 
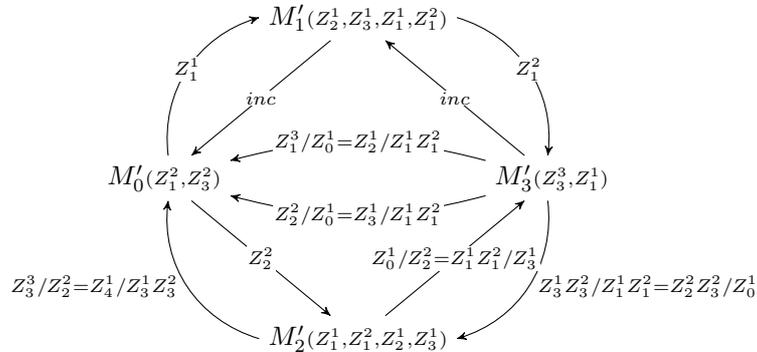
\begin{align*}
\begin{tikzpicture}[scale =0.85]
\node (C0) at (0,0)  {$M_0'\scriptstyle{(Z_1^2,Z_3^2)}$};
\node (C2) at (6,0)  {$M_3'\scriptstyle{(Z_3^3,Z_1^1)}$};
\node (C1) at (3,2.5)  {$M_1'\scriptstyle{(Z_2^1,Z_3^1,Z_1^1,Z_1^2)}$};
\node (C3) at (3,-2.5)  {$M_2'\scriptstyle{(Z_1^1,Z_1^2,Z_2^1,Z_3^1)}$};
\draw [->,bend left=40] (C0) to node[gap]  {$\scriptstyle{Z_1^{1}}$} (C1);
\draw [->,bend left=40] (C1) to node[gap]  {$\scriptstyle{Z_1^{2}}$} (C2);
\draw [->,bend left=40] (C2) to node[right]  {$\scriptstyle{Z_3^{1}Z_3^{2}/Z_1^{1}Z_1^{2}=Z_2^{2}Z_3^{2}/Z_0^{1}}$} (C3);
\draw [->,bend left=40] (C3) to node[left]  {$\scriptstyle{Z_3^{3}/Z_2^{2}=Z_4^{1}/Z_3^{1}Z_3^{2}}$} (C0);
\draw [->] (C1) to node[gap]  {$\scriptstyle{inc}$} (C0);
\draw [->] (C2) to node[gap]  {$\scriptstyle{inc}$} (C1);
\draw [->] (C3) to node[gap]  {$\scriptstyle{Z_0^{1}/Z_2^{2}=Z_1^{1}Z_1^{2}/Z_3^{1}}$} (C2);
\draw [->] (C0) to node[gap]  {$\scriptstyle{Z_2^{2}}$} (C3);
\draw [->,bend right=15] (C2) to node[gap]  {$\scriptstyle{Z_1^{3}/Z_0^{1}=Z_2^{1}/Z_1^{1}Z_1^{2}}$} (C0);
\draw [->,bend left=15] (C2) to node[gap]  {$\scriptstyle{Z_2^{2}/Z_0^{1}=Z_3^{1}/Z^{1}_1 Z_1^{2}}$} (C0);
\end{tikzpicture}
\end{align*}
with relations those implied by the labelling, where we have suppressed the upper parenthesis in the labelling and denoted labelled loops at a node by labels in parenthesis following the node in order to simplify the diagram.
\end{RCAExample}

\begin{ExampleSimpleTypeA}($G=\frac{1}{r}(1,1)$) As another example we consider the family of graded reconstruction algebras arising from the finite groups $\frac{1}{r}(1,1)$. In this case the Hirzebruch-Jung continued fractions are $r/1=[r]$ and $r/(r-1)=[2,2 \dots, 2]$. The invariant ring $\mathbb{C}[x,y]^{\frac{1}{r}(1,1)}=\mathbb{C}[x^r, x^{r-1}y, \dots , xy^{r-1},y^r]$ is isomorphic to the polynomial algebra $\mathbb{C}[Z_0, \dots, Z_{r}]$ subject to the relations given by all $2 \times 2$ minors of the matrix
\[
\left( \begin{array}{c c c c c}
Z_0 & Z_1 & \dots & Z_{r-2} & Z_{r-1} \\
Z_1 & Z_2 & \dots & Z_{r-1} & Z_{r} 
\end{array} \right).
\]
This is a graded algebra with $Z_i$ in degree $r$ and homogeneous relations generated in degree $2r$. In this case there are two indecomposable special maximal Cohen-Macaulay modules, $M_0=(1)$ and $M_1=(Z_0,Z_1) \cong (Z_i,Z_{i+1}) \cong (x,y)$. We note that these are graded by $M_0$ being generated in degree 0 and $M_1$ being generated in degree $1$. Then the reconstruction algebra can be described as the following quiver  
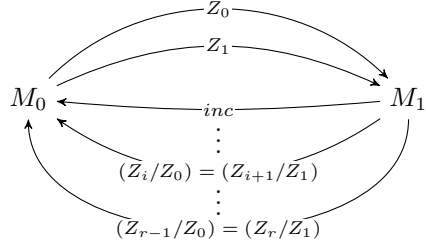
\begin{align*}
\begin{tikzpicture} [bend angle=45, looseness=1]
\node (C1) at (0,0)  {$M_0$};
\node (C2) at (5,0)  {$M_1$};
\node (C3) at (2.5,-0.45)  {$\vdots$};
\node (C3) at (2.5,-1.27)  {$\vdots$};
\draw [->,bend left] (C1) to node[gap]  {\scriptsize{$Z_0$}} (C2);
\draw [->,bend left=25] (C1) to node[gap] {\scriptsize{$Z_1$}} (C2);
\draw [->,bend left=90] (C2) to node[gap]  {\scriptsize{$(Z_{r-1}/Z_0) =(Z_{r}/Z_1)$}} (C1);
\draw [->,bend left=5] (C2) to node[gap] {\scriptsize{$inc$}} (C1);
\draw [->,bend left=35] (C2) to node[gap] {\scriptsize{$(Z_i/Z_0)=(Z_{i+1}/Z_1)$}} (C1);
\end{tikzpicture}
\end{align*}
with relations implied by the labelling, and where the arrows from $0$ to $1$ have degree 1, the arrows from $1$ to $0$ have degree $r-1$, and the relations have homogeneous generators in degree $r$. 

The Artin component of the versal deformation of the singularity (after base change) is defined by $\mathbb{C}[Z^{(1)}_0,Z_1^{(1)},Z_1^{(2)}, \dots, Z^{(1)}_{r}]$ subject to the relations given by all $2 \times 2$ minors of the matrix
\[
\left( \begin{array}{c c c c c}
Z^{(1)}_0 & Z^{(2)}_1 & \dots & Z^{(2)}_{r-2} & Z^{(2)}_{r-1} \\
Z^{(1)}_1 & Z^{(1)}_2 & \dots & Z^{(1)}_{r-1} & Z^{(1)}_{r} 
\end{array} \right).
\]
The deformed reconstruction algebra  is then the endomorphism algebra of the two modules $M_0'=(1)$ and $M_1'=\left(Z_0^{(1)},Z_1^{(1)}\right)$ which lift $M_0$ and $M_1$, and $A$ can be described as the path algebra of the following quiver
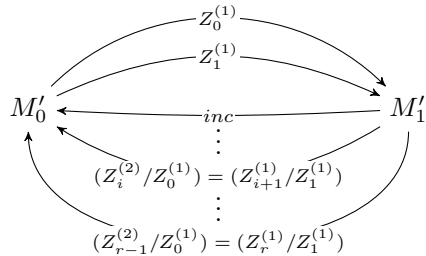
\begin{center}
 \begin{align*}
\begin{tikzpicture} [bend angle=45, looseness=1]
\node (C1) at (0,0)  {$M_0'$};
\node (C2) at (5,0)  {$M_1'$};
\node (C3) at (2.5,-0.33)  {$\vdots$};
\node (C3) at (2.5,-1.22)  {$\vdots$};
\draw [->,bend left] (C1) to node[gap]  {\scriptsize{$Z^{(1)}_0$}} (C2);
\draw [->,bend left=25] (C1) to node[gap] {\scriptsize{$Z^{(1)}_1$}} (C2);
\draw [->,bend left=90] (C2) to node[gap]  {\scriptsize{$(Z^{(2)}_{r-1}/Z^{(1)}_0) =(Z^{(1)}_r/Z^{(1)}_1)$}} (C1);
\draw [->,bend left=3] (C2) to node[gap] {\scriptsize{$inc$}} (C1);
\draw [->,bend left=32] (C2) to node[gap] {\scriptsize{$(Z^{(2)}_{i}/Z^{(1)}_0) =(Z^{(1)}_{i+1}/Z^{(1)}_1)$}} (C1);
\end{tikzpicture}
\end{align*}
\end{center}
with relations implied by the labelling. We note that the arrows from $0$ to $1$ still have degree 1 and  the arrows from $1$ to $0$ still have degree $r-1$ , but the relations now have homogeneous generators in degree $r+1$.
\end{ExampleSimpleTypeA}

When $r=2$ this deformed reconstruction algebra is the conifold quiver. More generally the versal  deformation of the minimal resolution can be realised as the total bundle of $\bigoplus_{i=1}^{r} \mathcal{O}_{\mathbb{P}^1}(-1)$ and the tilting bundle defining the deformed reconstruction algebra can be lifted from the tilting bundle $\mathcal{O}_{\mathbb{P}^1} \oplus \mathcal{O}_{\mathbb{P}^1}(1)$ on  $\mathbb{P}^1$. In this situation the calculations could be carried out explicitly on the minimal resolution rather than the base singularity.

\subsection{Examples: outside of type $\mathbb{A}$}
We calculate two explicit examples outside of the type $\mathbb{A}$ case; one example corresponding to a type $\mathbb{D}$ quotient singularity and one example corresponding to a non-quotient rational surface singularity. As we work outside of the quotient singularity setting in these examples we will work complete locally.

\begin{TypeDExample}(Type $\mathbb{D}_{5,2}$)
Consider the non-abelian group of type $\mathbb{D}_{5,2}$
\begin{equation*}
G:= \left\langle \left( \begin{array}{c c} \varepsilon_4 & 0 \\ 0 & \varepsilon_4^{-1} \end{array} \right) , \left( \begin{array}{c c} 0 &\varepsilon_4 \\ \varepsilon_4 & 0 \end{array} \right), \left( \begin{array}{c c} 0 &\varepsilon_6 \\ \varepsilon_6 & 0 \end{array} \right) \right \rangle
\end{equation*}
which is the direct product of a cyclic group of order 3 with the quaternion group of order 8. In this case the corresponding singularity $\mathbb{C}[[x,y]]^G$ is isomorphic to the polynomial ring $\mathbb{C}[[t,X_1,X_2,X_3,Y_1,Y_2,Y_3]]$ subject to the relations generated by the 2 by 2 minors of the matrix
\[
\left( \begin{array}{c c c }
X_1 & X_2 & X_3  \\
Y_1 & Y_2 & Y_3
\end{array} \right),
\]
$X_1=t^2$, $Y_2=t^2$, and $X_3-Y_3=t^2$, see \cite[Corollary 3.6]{WahlEqnsRational}. The minimal resolution of this singularity has an exceptional divisor which consists of a central -3 curve intersecting 3 outer -2 curves, and so has a dual graph of Dynkin type $\mathbb{D}_4$.

The special MCM modules can be described as $M_0=(1)$, $M_+=(X_2,t)$, $M_-=(t,Y_1)$, $M_1=(X_3,Y_3) \cong (X_2,Y_2) \cong (X_1,Y_1)$, and $M_2=(X_3,Y_3,t) $, and then the reconstruction algebra can be presented as the completion of the path algebra of the following quiver
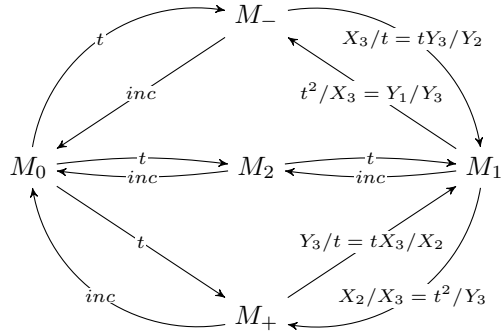
\begin{center}
\begin{tikzpicture} [bend angle=18, looseness=1, scale=1]
\node (C1) at (0,0)  {$M_0$};
\node (C2) at (3,-2)  {$M_+$};
\node (C3) at (3,0){$M_2$};
\node (C4) at (3,2) {$M_-$};
\node (C5) at (6,0) {$M_1$};
\draw [->,bend left=0] (C1) to node[gap]  {\scriptsize{$t$}} (C2);
\draw [->,bend left=45] (C2) to node[gap]  {\scriptsize{$inc$}} (C1);
\draw [->,bend left=0] (C2) to node[gap]  {\scriptsize{$Y_3/t=tX_3/X_2$}} (C5);
\draw [->,bend left=45] (C5) to node[gap]  {\scriptsize{$X_2/X_3=t^2/Y_3$}} (C2);
\draw [->,bend left=7] (C1) to node[gap]  {\scriptsize{$t$}} (C3);
\draw [->,bend left=7] (C3) to node[gap]  {\scriptsize{$inc$}} (C1);
\draw [->,bend left=7] (C3) to node[gap]  {\scriptsize{$t$}} (C5);
\draw [->,bend left=7] (C5) to node[gap]  {\scriptsize{$inc$}} (C3);
\draw [->,bend left=45] (C1) to node[gap]  {\scriptsize{$t$}} (C4);
\draw [->,bend left=0] (C4) to node[gap]  {\scriptsize{$inc$}} (C1);
\draw [->,bend left=45] (C4) to node[gap]  {\scriptsize{$X_3/t=tY_3/Y_2$}} (C5);
\draw [->,bend left=0] (C5) to node[gap]  {\scriptsize{$t^2/X_3=Y_1/Y_3$}} (C4);
\end{tikzpicture}
\end{center}
with relations those implied by the labelling. See \cite{RCAD1} for a more complete description of reconstruction algebras of this type.

The Artin component of the versal deformation of the singularity can be described (after base change) as the polynomial ring $\mathbb{C}[[t_1,t_1',t_2,t_2',t_3,t_3',X_1,X_2,X_3,Y_1,Y_2,Y_3]]$ subject to the relations generated by the 2 by 2 minors of the matrix
\[
\left( \begin{array}{c c c }
X_1 & X_2 & X_3 \\
Y_1 & Y_2 & Y_3
\end{array} \right),
\]
$X_1=t_1t_1'$, $Y_2=t_2t_2'$, and $X_3-Y_3=t_3t_3'$, see \cite[Theorem 2]{WahlSimultaneous}. The action of the product of Weyl groups $S_2 \times S_2 \times S_2$ is clear to see on the pairs $t_i,t_i'$. Then the lifts of the $M_i$ to the deformation are $M_0'=(1)$, $M_+'=(X_2,t_1)$, $M_-'=(t_2,Y_1)$, $M_1'=(X_3,Y_3) \cong (X_1,Y_1) \cong (X_2,Y_2)$, and $M_2'=(X_3,Y_3,t_3)$,
and the deformation of the reconstruction algebra can be presented as the completion of the path algebra of the following quiver
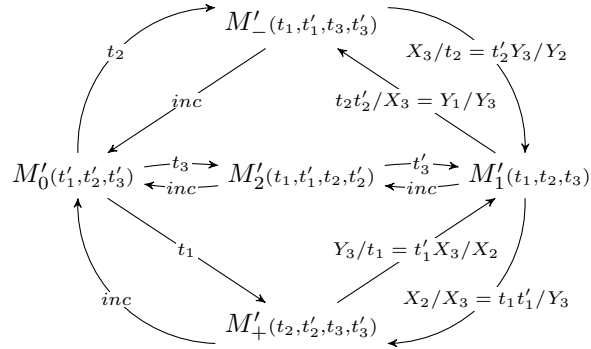
\begin{center}
\begin{tikzpicture} [bend angle=15, looseness=1, scale=1]
\node (C1) at (0,0)  {$M_0'\scriptstyle{(t_1',t_2',t_3')}$};
\node (C2) at (3,-2)  {$M_+'\scriptstyle{(t_2,t_2',t_3,t_3')}$};
\node (C3) at (3, 0){$M_2'\scriptstyle{(t_1,t_1',t_2,t_2')}$};
\node (C4) at (3,2) {$M_-'\scriptstyle{(t_1,t_1',t_3,t_3')}$};
\node (C5) at (6,0) {$M_1'\scriptstyle{(t_1,t_2,t_3)}$};
\draw [->,bend left=0] (C1) to node[gap]  {\scriptsize{$t_1$}} (C2);
\draw [->,bend left=45] (C2) to node[gap]  {\scriptsize{$inc$}} (C1);
\draw [->,bend left=0] (C2) to node[gap]  {\scriptsize{$Y_3/t_1=t_1'X_3/X_2$}} (C5);
\draw [->,bend left=45] (C5) to node[gap]  {\scriptsize{$X_2/X_3=t_1t_1'/Y_3$}} (C2);
\draw [->,bend left=7] (C1) to node[gap]  {\scriptsize{$t_3$}} (C3);
\draw [->,bend left=7] (C3) to node[gap]  {\scriptsize{$inc$}} (C1);
\draw [->,bend left=7] (C3) to node[gap]  {\scriptsize{$t_3'$}} (C5);
\draw [->,bend left=7] (C5) to node[gap]  {\scriptsize{$inc$}} (C3);
\draw [->,bend left=45] (C1) to node[gap]  {\scriptsize{$t_2$}} (C4);
\draw [->,bend left=0] (C4) to node[gap]  {\scriptsize{$inc$}} (C1);
\draw [->,bend left=45] (C4) to node[gap]  {\scriptsize{$X_3/t_2=t_2'Y_3/Y_2$}} (C5);
\draw [->,bend left=0] (C5) to node[gap]  {\scriptsize{$t_2t_2'/X_3=Y_1/Y_3$}} (C4);
\end{tikzpicture}
\end{center}
with relations those implied by the labelling, where we have denoted loops at a node by the labellings of the loops in parenthesis to simplify the diagram.
\end{TypeDExample}

\begin{NonQuotientExample}(Non-quotient) Consider the non-quotient rational surface singularity corresponding to the polynomial ring $\mathbb{C}[[t,X_1,X_2,X_3,X_4,Y_1,Y_2,Y_3,Y_4]]$ subject to the relations generated by the 2 by 2 minors of 
\[
\left( \begin{array}{c c c c }
 X_1 & X_2 & X_3 & X_4  \\
Y_1 & Y_2 & Y_3 & Y_4 
\end{array} \right),
\]
$X_1=t^2$, $Y_2=t^2$, $X_3-Y_3=t^2$, and $X_4 -\lambda Y_4=t^2$ for some $\lambda \in \mathbb{C}\backslash \{0,1\}$. For any $\lambda \in \mathbb{C}\backslash \{0,1\}$ this describes a rational surface singularity whose minimal resolution has an exceptional divisor consisting of a central -4 curve intersecting 4 outer -2 curves, see \cite[Corollary 3.6]{WahlEqnsRational}.

The indecomposable special MCM modules for this singularity can be described as $M_0=(1)$, $M_1=(t,Y_1)$, $M_2=(X_2,t)$, $M_3=(X_3,Y_3,t)$, $M_4=(X_4,Y_4,t)$ and $M_5=(X_3,Y_3)$, and the reconstruction can be presented as the completion of the path algebra of the quiver
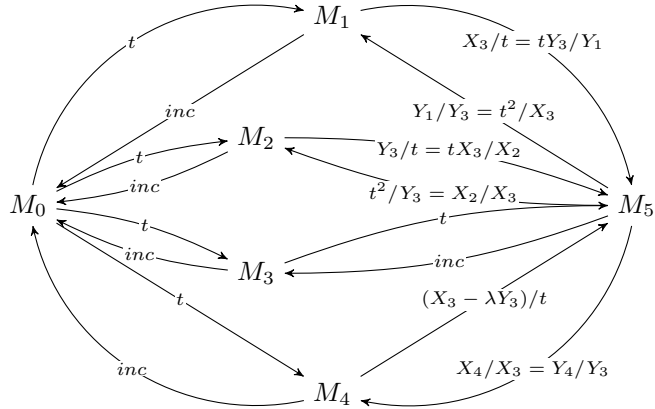
\begin{center}
\begin{tikzpicture} [bend angle=10, looseness=1, scale=1]
\node (C1) at (0,0)  {$M_0$};
\node (C2) at (4,2.5)  {$M_1$};
\node (C3) at (3, 0.9){$M_2$};
\node (C4) at (3,-0.9) {$M_3$};
\node (E) at (4,-2.5) {$M_4$};
\node (C5) at (8,0) {$M_5$};
\draw [->,bend left=45] (C1) to node[gap]  {\scriptsize{$t$}} (C2);
\draw [->,bend left=0] (C2) to node[gap]  {\scriptsize{$inc$}} (C1);
\draw [->,bend left=45] (C2) to node[gap]  {\scriptsize{$X_3/t=tY_3/Y_1$}} (C5);
\draw [->,bend left=0] (C5) to node[gap]  {\scriptsize{$Y_1/Y_3=t^2/X_3$}} (C2);
\draw [->,bend left] (C1) to node[gap]  {\scriptsize{$t$}} (C3);
\draw [->,bend left] (C3) to node[gap]  {\scriptsize{$inc$}} (C1);
\draw [->,bend left] (C3) to node[gap]  {\scriptsize{$Y_3/t=tX_3/X_2$}} (C5);
\draw [->,bend left] (C5) to node[gap]  {\scriptsize{$t^2/Y_3=X_2/X_3$}} (C3);
\draw [->,bend left] (C1) to node[gap]  {\scriptsize{$t$}} (C4);
\draw [->,bend left] (C4) to node[gap]  {\scriptsize{$inc$}} (C1);
\draw [->,bend left] (C4) to node[gap]  {\scriptsize{$t$}} (C5);
\draw [->,bend left] (C5) to node[gap]  {\scriptsize{$inc$}} (C4);
\draw [->,bend left=0] (C1) to node[gap]  {\scriptsize{$t$}} (E);
\draw [->,bend left=45] (E) to node[gap]  {\scriptsize{$inc$}} (C1);
\draw [->,bend left=0] (E) to node[gap]  {\scriptsize{$(X_3-\lambda Y_3)/t$}} (C5);
\draw [->,bend left=45] (C5) to node[gap]  {\scriptsize{$X_4/X_3=Y_4/Y_3$}} (E);
\end{tikzpicture}
\end{center}
with relations those implied by the labelling, where we note that $(X_3-\lambda Y_3)/t = tX_3/X_4=tY_3/Y_4$.

The Artin component of the versal deformation of the singularity (after base change) is defined by the polynomial ring $\mathbb{C}[[t_1,t_1',t_2,t_2',t_3,t_3',t_4,t_4',X_1,X_2,X_3,X_4,Y_1,Y_2,Y_3,Y_4]]$ subject to the relations generated by the 2 by 2 minors of 
\[
\left( \begin{array}{c c c c }
X_1 & X_2 & X_3 & X_4  \\
Y_1 & Y_2 & Y_3 & Y_4 
\end{array} \right),
\]
$X_1=t_1t_1'$, $Y_2=t_2t_2'$, $X_3-Y_3=t_3t_3'$, and $X_4-\lambda Y_4=t_4t_4'$, see \cite[Theorem 2]{WahlSimultaneous}. Again, it is clear to see the action of the product of Weyl groups $S_2 \times S_2 \times S_2 \times S_2$ permuting the pairs $t,t'$.

The lifts of the $M_i$ to the versal deformation of the minimal resolution can be described as  $M_0'=(1)$, $M_1'=(t_1,Y_1)$, $M_2'=(X_2,t_2)$, $M_3'=(X_3,Y_3,t_3)$, $M_4'=(X_4,Y_4,t_4)$ and $M_5'=(X_3,Y_3)$, and the deformed reconstruction algebra can be presented as the completion of the path algebra of the quiver
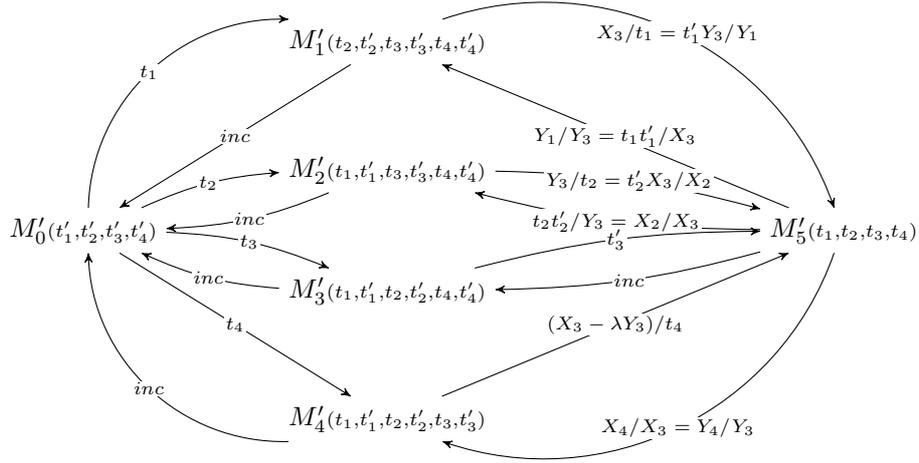
\begin{center}
\begin{tikzpicture} [bend angle=10, looseness=1, scale=1]
\node (C1) at (0,0)  {$M'_0\scriptstyle{(t_1',t_2',t_3',t_4')}$};
\node (C2) at (4,2.5)  {$M'_1\scriptstyle{(t_2,t_2',t_3,t_3',t_4,t_4')}$};
\node (C3) at (4, 0.8){$M'_2\scriptstyle{(t_1,t_1',t_3,t_3',t_4,t_4')}$};
\node (C4) at (4,-0.8) {$M'_3\scriptstyle{(t_1,t_1',t_2,t_2',t_4,t_4')}$};
\node (E) at (4,-2.5) {$M'_4\scriptstyle{(t_1,t_1',t_2,t_2',t_3,t_3')}$};
\node (C5) at (10,0) {$M'_5\scriptstyle{(t_1,t_2,t_3,t_4)}$};
\draw [->,bend left=45] (C1) to node[gap]  {\scriptsize{$t_1$}} (C2);
\draw [->,bend left=0] (C2) to node[gap]  {\scriptsize{$inc$}} (C1);
\draw [->,bend left=45] (C2) to node[gap]  {\scriptsize{$X_3/t_1=t_1'Y_3/Y_1$}} (C5);
\draw [->,bend left=0] (C5) to node[gap]  {\scriptsize{$Y_1/Y_3=t_1t_1'/X_3$}} (C2);
\draw [->,bend left] (C1) to node[gap]  {\scriptsize{$t_2$}} (C3);
\draw [->,bend left] (C3) to node[gap]  {\scriptsize{$inc$}} (C1);
\draw [->,bend left=7] (C3) to node[gap]  {\scriptsize{$Y_3/t_2=t_2'X_3/X_2$}} (C5);
\draw [->,bend left=7] (C5) to node[gap]  {\scriptsize{$t_2t_2'/Y_3=X_2/X_3$}} (C3);
\draw [->,bend left] (C1) to node[gap]  {\scriptsize{$t_3$}} (C4);
\draw [->,bend left] (C4) to node[gap]  {\scriptsize{$inc$}} (C1);
\draw [->,bend left=7] (C4) to node[gap]  {\scriptsize{$t_3'$}} (C5);
\draw [->,bend left=7] (C5) to node[gap]  {\scriptsize{$inc$}} (C4);
\draw [->,bend left=0] (C1) to node[gap]  {\scriptsize{$t_4$}} (E);
\draw [->,bend left=45] (E) to node[gap]  {\scriptsize{$inc$}} (C1);
\draw [->,bend left=0] (E) to node[gap]  {\scriptsize{$(X_3-\lambda Y_3)/t_4$}} (C5);
\draw [->,bend left=45] (C5) to node[gap]  {\scriptsize{$X_4/X_3=Y_4/Y_3$}} (E);
\end{tikzpicture}
\end{center}
with relations those implied by the labelling, where we denote labelled loops at a node by labels in parenthesis at that node containing the labelling of the loops. Again, we note that $(X_3-\lambda Y_3)/t_4 =t_4'Y_3/Y_4=t_4'X_3/X_4$.
\end{NonQuotientExample}

\subsection{Non-geometric deformations and quotient surface singularities} Theorem \ref{Graded Main} constructs graded deformations that correspond to geometric deformations \label{NonGeometric} for noncommutative algebras that occur as the endomorphism algebra of a tilting bundle. In general there are further deformations of the noncommutative algebra which do not correspond to geometric deformations, such as occur for the preprojective algebras \cite{CBH}. However, for graded reconstruction algebras corresponding to $G<\GL_2(\mathbb{C})$ not contained in $\SL_2(\mathbb{C})$ all graded deformations correspond to geometric ones; there is no analogue to the parameter $t$ outside of $\SL_2(\mathbb{C})$. This is as graded deformations of a graded reconstruction algebra $A$ are controlled by the graded Hochschild cohomology group $\HH^2_{(<0)}(A)$ (see \cite[Section 1 and 2]{BrG}), the geometric deformations are controlled by $\CH^1(X,\mathcal{T}_{X})$ where $\mathcal{T}_{X}$ is the tangent sheaf (see \cite[Theorem 5.2, Corollary 10.3]{HartDef}), and the following result compares these two cohomology groups.

\begin{GL2AllDeforms} \label{GL2AllDeforms} If $X\rightarrow \Spec(R)$ is the minimal resolution of a surface quotient singularity $R=\mathbb{C}[x,y]^G$ for $G<\GL_2(\mathbb{C})$ a small, finite subgroup and $A$ is the corresponding graded reconstruction algebra, then
\begin{equation*}
\HH^2_{(<0)}(A) = \left\{ \begin{array}{c c} \CH^1(X,\mathcal{T}_X) & \text{ if } G \nless \SL_{2}(\mathbb{C}) \\
                                                                         \CH^1(X,\mathcal{T}_X) \oplus \mathbb{C}\Omega & \text{ if } G < \SL_2(\mathbb{C}) \end{array} \right.
\end{equation*}
where $\Omega \in \CH^0(X,\wedge^2 \mathcal{T}_X)$ corresponds to the symplectic form on $X$ induced by $\SL_2(\mathbb{C}) \cong \Sp_2(\mathbb{C})$.
\end{GL2AllDeforms}
\begin{proof}
By \cite[Corollary 3.5]{BHTilting}  $\HH^2(A) \cong \HH^2(X)$, by \cite[Corollary 2.6]{SwanHoch} $\HH^2(X) \cong \CH^0(X,\wedge^2 \mathcal{T}_X) \oplus \CH^1(X, \mathcal{T}_X) \oplus \CH^2(X,\mathcal{O}_X)$, and as $\Spec(R)$ is a rational singularity $\CH^2(X,\mathcal{O}_X)=0$. Due to the $\mathbb{C}^*$-actions $\HH_{(<0)}(A) \cong \CH^1_{(<0)}(X,\mathcal{T}_X) \oplus \CH^0_{(<0)}(X,\wedge^2 \mathcal{T}_X)$ considered as graded $R$-modules, and we note that $\CH^1(X,\mathcal{T}_X)=\CH^1_{(<0)}(X,\mathcal{T}_X)$ by Lemma \ref{HHSQS}. Hence we need only consider $\CH_{(<0)}(X,\wedge^2 \mathcal{T}_X)$.

In particular, $\wedge^2 \mathcal{T}_X \cong \omega_X^{\vee}$ and $\CH^0(X,\omega_{X}^{\vee}) \cong \Hom_R(\omega_{R},R)$ by \cite[Lemma 2.2]{Esnault}. The canonical module $\omega_R$  corresponds to $(\mathbb{C}[x,y] \, dx \wedge dy)^G$, see \cite[Section 3]{Aus}, hence is generated  as an $\mathbb{C}[x,y]^G$-module by a single element of degree $2$. Then the vector space $\Hom_{R}(\omega_R,R)_{(<0)}=0$ is equal to the vector space of paths of length $<2$ in the McKay quiver between the determinant representation and the trivial representation, which correspond to $\omega_R$ and $R$ respectively. 

There is a path of length zero if and only if $R \cong \omega_R$, which is precisely the case $G< \SL_2(\mathbb{C})$. In this case the isomorphism between $\omega_R$ and $R$ must send $dx \wedge dy$ to a scalar multiple of $1$, so has degree $-2$ and this element corresponds to a scalar multiple of $\Omega$.

Next, we need to show there are no length one paths (arrows) between $\omega_R$ and $R$.  As $G$ is small, the Auslander-Reiten (AR) quiver equals the McKay quiver \cite[Section 2]{Aus}.  The arrows in the AR quiver can be calculated using the fundamental sequence 
\begin{equation*}
0 \rightarrow \omega_R \rightarrow E \rightarrow R
\end{equation*}
where $E$ is the Auslander module and the arrows from $\omega_R$ can only go to the indecomposable summands of $E$. In particular, if $G$ is not cyclic the Auslander module is indecomposable of rank 2 \cite[Theorem 2.1]{YoshinoKawamoto} and hence there are no arrows between $\omega_R$ and $R$ as $R$ has rank one. When $G$ is cyclic $E$ is decomposable, but also known to not contain a copy of $A$ when $G$ is small \cite[Section 3]{YoshinoKawamoto}. Hence $\Hom_R(\omega_R,R)$ contains no paths of length one.
\end{proof}

\subsection{Symplectic quotient singularities}  \label{Deformations of Symplectic Quotient Singularites} As a second application we consider crepant resolutions of symplectic quotient singularities. We will recall the existence first of $\mathbb{C}^*$-equivariant deformations and then of $\mathbb{C}^*$-equivariant tilting bundles in order to apply Theorem \ref{Graded Main}. In this case we will produce graded deformations related to the symplectic reflection algebras.

Let $V=\mathbb{C}^{2n}$ with symplectic form $\omega$, let $G< \Sp(V)$ be a finite subgroup generated by symplectic reflections and acting faithfully on $V$ such that there is no $G$-invariant symplectic splitting, let $\Spec(R_0)=V/G$ be the corresponding quotient singularity, and suppose that a crepant resolution of singularities $p: X_0 \rightarrow \Spec(R_0)$ exists. We note that crepant resolutions only exist for some symplectic quotient singularities. Then $X_0$ carries a $\mathbb{C}^*$-action induced from that on $\mathbb{C}^{2n}$ such that $p:X_0 \rightarrow \mathbb{C}^{2n}/G$ is $\mathbb{C}^*$-equivariant \cite[Theorem 1.3 ii)]{CrepantResAction}.

Symplectic resolutions are known to have a class of good $\mathbb{C}^*$-equivariant deformations and in particular there is the following result.
\begin{KaledinGinzburg}[{\cite[Theorems 1.13 and 1.16]{PoisDefKaledinGinzburg}}] \label{KaledinGinzburg} There exists a smooth variety $X$ with $\mathbb{C}^*$-action and a flat $\mathbb{C}^*$-equivariant morphism  $\rho:X \rightarrow D=\Spec(\mathbb{C}[\beta_1, \dots, \beta_k])$ such that $X_0$ is the fibre of $\rho$ over the unique $\mathbb{C}^*$-invariant point $d=(\beta_1, \dots, \beta_k)$. The $\mathbb{C}^*$-action on $D$ is given by each $\beta_j$ having degree $2$ and $k$ is the number of conjugacy classes of symplectic reflections in $G$.   In particular $D$ is a graded complete local $\mathbb{C}$-algebra of finite type with unique graded maximal ideal $d$ and $(\rho: X \rightarrow \Spec(D), j_d:X_0\rightarrow X)$ provides a $\mathbb{C}^*$-equivariant deformation $X_0$. Moreover, this morphism factors  though a projective birational $\mathbb{C}^*$-equivariant morphism $\pi:X \rightarrow \Spec(R)$ where $R$ is also flat over $D$ with good $\mathbb{C}^*$-action and $R \otimes_D D/d \cong R_0.$
\end{KaledinGinzburg}
This deformation is known as the $\mathbb{C}^*$-equivariant versal Poisson deformation.

We also recall the existence of tilting bundles on symplectic resolutions.
\begin{SymplecticTiltingBundle}[{\cite[Theorems 2.3 and 4.3]{SympMcKay}}] \label{SymplecticTiltingBundle}
There exists a tilting bundle $T_0$ on $X_0$ tilting to the smash product algebra $\End_{X_0}(T_0) \cong \mathbb{C}[V] \rtimes G$, and there is a canonical $\mathbb{C}^*$-equivariant structure on $T_0$ such that $\mathbb{C}[V] \rtimes G$ is graded with $\mathbb{C}G$ in degree 0 and linear elements of $\mathbb{C}[V]=\Sym(V^*)$ in degree 1.
\end{SymplecticTiltingBundle}

As there exist $\mathbb{C}^*$-equivariant deformations of a symplectic resolution and $\mathbb{C}^*$-equivariant tilting bundles we can apply Theorem \ref{Graded Main} to produce graded deformations of the smash product algebra.

\begin{SymplectRef} \label{SymplectRef}
 Consider the tilting bundle $T_0$ from Lemma \ref{SymplecticTiltingBundle}, and let $(\rho:X \rightarrow \Spec \, D,j_d:X_0 \rightarrow X)$ be the $\mathbb{C}^*$-equivariant deformation of Lemma \ref{KaledinGinzburg}. Then:
\begin{enumerate}
\item The $\mathbb{C}^*$-equivariant tilting bundle $T_0$ on $X_0$ lifts to a $\mathbb{C}^*$-equivariant tilting bundle $T$ on $X$.
\item The algebra $A:=\End_{X}(T)$ is derived equivalent to the $\mathbb{C}^*$-equivariant versal Poisson deformation space $X$.
\item The algebra $A$ is a graded deformation of the skew group algebra over $D$ induced by the $\mathbb{C}^*$-equivariant deformation $(\rho: X \rightarrow \Spec \, D , j_d:X_0 \rightarrow X)$.
\item For any closed point $z \in \Spec(D)$ the fibres $A_z:=A \otimes_D (D/z)$ are symplectic reflection algebras with noncommutative parameter $t$ set equal to 0.
\end{enumerate}
\end{SymplectRef}
Part (4) follows as the symplectic reflection algebras are defined to be PBW deformations of  the skew group algebras, and PBW deformations are exactly the filtered algebras appearing as fibres of graded deformations. The corollary above recovers exactly those PBW deformations that correspond to deformations of the crepant resolution, however there is a further parameter $t$ for a general symplectic reflection algebra which captures the deformation induced by the symplectic form and does not correspond to a geometric deformation.

In particular, when a sympletic quotient singularity has a crepant resolution this result produces a derived equivalence between symplectic reflection algebras with parameter $t=0$ and fibres of the $\mathbb{C}^*$-equivariant Poisson deformation of the resolution.  In the case of wreath product groups, where crepant resolutions are known to exist, this was already known due to a result of Gordon and Smith \cite{GordonSmithRepresentationsSRA}. Their result proves that a symplectic reflection algebra with parameter $t=0$ arising from a wreath product algebra can be produced from a tilting bundle on a crepant resolution of the singularity defined by centre of this symplectic reflection algebra. The above corollary recaptures this via a different method: in the language of Corollary \ref{FilteredCorollary} the symplectic reflection algebra is $A_z$, its centre can be realised as $R_z$, the variety $X_z$ is a crepant resolution of the corresponding singularity, and $A_z$ is indeed defined by a tilting bundle $T_z$ on $X_z$.

\bibliographystyle{alpha}
\bibliography{RCABIB}

\end{document}